\documentclass[a4paper,11pt]{article}

\usepackage{amsthm}
\usepackage{latexsym}
\usepackage{dsfont}
\usepackage{bbm}
\usepackage{amssymb}
\usepackage{amsmath}
\usepackage{graphicx}

\newtheorem{Theorem}{Theorem}[section]
\newtheorem{Lemma}[Theorem]{Lemma}
\newtheorem{Cor}[Theorem]{Corollary}
\newtheorem{Rem}[Theorem]{Remark}

\newtheorem{Def}[Theorem]{Definition}
\newtheorem{Prop}[Theorem]{Proposition}

\def\gl{\buildrel \rm def\over =}

\DeclareMathOperator{\Weib}{\mathrm{Weibull}}
\DeclareMathOperator{\perWeib}{\mathrm{-Weibull}}
\DeclareMathOperator{\Exp}{\mathrm{Exp}}
\DeclareMathOperator{\esssup}{ess \, sup}

\DeclareMathOperator{\G}{\mathbb{G}}
\DeclareMathOperator{\V}{\mathbb{V}}

\title{\bf A Min-Type Stochastic Fixed-Point Equation Related to the Smoothing Transformation}

\author{\sc Gerold Alsmeyer and Matthias Meiners\\
    \\
    Institut f\"ur Mathematische Statistik \\
    Fachbereich Mathematik \\
    Einsteinstra\ss e 62 \\
    D-48149 M\"unster, Germany}

\begin{document}
\maketitle

\begin{abstract}
\noindent
This paper is devoted to the study of the stochastic fixed-point equation
\begin{equation*}
X\ \stackrel{d}{=}\ \inf_{i \geq 1: T_i > 0} X_i/T_i
\end{equation*} 
and the connection with its additive counterpart $X\stackrel{d}{=}\sum_{i\ge 1}T_{i}X_{i}$ associated with the smoothing transformation. Here $\stackrel{d}{=}$ means equality in distribution, $T\gl (T_i)_{i \geq 1}$ is a given sequence of nonnegative random variables and $X, X_1, \ldots$ is a sequence of nonnegative i.i.d.\ random variables independent of $T$. We draw attention to the question of the existence of nontrivial solutions and, in particular, of special solutions named $\alpha$-regular solutions $(\alpha>0)$. We give a complete answer to the question of when $\alpha$-regular solutions exist and prove that they are always mixtures of Weibull distributions or certain periodic variants. We also give a complete characterization of all fixed points of this kind. A disintegration method which leads to the study of certain multiplicative martingales and a pathwise renewal equation
after a suitable transform are the key tools for our analysis. Finally, we provide corresponding results for the fixed points of the related additive equation mentioned above. To some extent, these results have been obtained earlier by Iksanov \cite{Ik}.
\vspace{0,1cm}


\hphantom{2000 Mathematics Subject Classification:	}39B22

\end{abstract}

\section{Introduction}
For a given sequence $T \gl (T_i)_{i \geq 1}$ of nonnegative random variables on a probability space $(\Omega,\mathfrak{A},\mathds{P})$ with $\sup_{i\ge 1}T_{i}>0$ a.s., consider the stochastic fixed-point equation (SFPE)
\begin{equation}\label{MinFP1}
X\ \stackrel{d}{=}\ \inf_{i \ge 1} X_i/T_i
\end{equation}
where $X,X_{1},X_{2},\ldots$ are i.i.d., nonnegative and independent of $T$ and
where $X_{i}/T_{i}\gl\infty$ is stipulated on $\{T_{i}=0\}$. A distribution $F$ on $[0,\infty)$ is called a solution to (\ref{MinFP1}) if this 
equation holds true with $X\stackrel{d}{=}F$, and it is called positive if $F(\{0\})=0$.  
Note that $F=\delta_{0}$, $\delta_{0}$ the Dirac measure at 0, always provides a trivial solution. The set of all solutions $\ne\delta_{0}$ will be denoted as $\mathfrak{F}_{\wedge}$ hereafter, or as $\mathfrak{F}_{\wedge}(T)$ if we want to emphasize its dependence on $T$.
We will make no notational distinction between a distribution $F$ and its left continuous distribution function,
and we denote by $\overline{F}$ the associated survival function, \textit{i.e.}, $\overline{F}\gl 1-F$. For
$F\in\mathfrak{F}_{\wedge}$, Eq.\ (\ref{MinFP1}) may be rewritten in terms of $\overline{F}$ as
\begin{equation}\label{MinFP2}
\overline{F}(t)\ =\ \mathds{E} \prod_{i \geq 1} \overline{F}(tT_i)
\end{equation}
for all $t \ge 0$.
Denote by $\mathcal{P},\overline{\mathcal{P}}$  the spaces of probability measures on $[0,\infty)$ and $[0,\infty]$, respectively. Defining the map $M: \mathcal{P}\rightarrow \mathcal{P}$ by
\begin{equation}\label{MapFP}
M(F)\  \gl\  \mathds{P} \left(\inf_{i \ge 1} \frac{X_i}{T_i} \in \, \cdot \right), \quad 
X\stackrel{d}{=} F,
\end{equation}
we see that, formally speaking, $\mathfrak{F}_{\wedge}$ is nothing but the set of fixed points $\ne
\delta_{0}$ of $M$, that is $\mathfrak{F}_{\wedge} = \{F \in \mathcal{P}: \, M(F) = F\} \setminus 
\{\delta_0\}$.

SFPEs of type \eqref{MinFP1} or similar with a min- or max-operation involved turn up in various fields of applied probability like probabilistic combinatorial optimization, the run-time analysis of divide-and-conquer algorithms or branching particle systems, where they typically characterize the asymptotic distribution of some random variable of interest (see \textit{e.g.} \cite{AldBan}, \cite{AldSt} and \cite{NR1}).  In particular, the probabilistic worst-case analysis of Hoare's \ttfamily FIND \rmfamily algorithm leads to the following fixed-point equation for the distributional limit $X$ of the linearly scaled maximal number of key comparisons of \ttfamily{FIND}.\rmfamily
\begin{equation*}	\label{eq:Exa1}
X	~\stackrel{d}{=}~ 1 + \max(UX_1,(1-U)X_2), 
\end{equation*}
where $U$ is a uniform $[0,1]$ random variable and $X_1, X_2$ are independent copies of the random variable $X$ (see \cite[Theorem 1]{Dev}). More general, in the analysis of divide-and-conquer algorithms, equations of the form
\begin{equation*}	\label{eq:Exa2}
X	~\stackrel{d}{=}~ \max_{i=1,\ldots,K} (A_i X_i + b_i)
\end{equation*}
appear (\textit{cf.}\ \cite{NR1} and \cite{R1}). Further, from a species competition model, the following fixed-point equation
\begin{equation*}	\label{eq:Exa3}
X	~\stackrel{d}{=}~ \eta + c \max_{i \geq 1} e^{-\xi_i} X_i,
\end{equation*}
arises, where $(\xi_i)_{i \geq 1}$ are the points of a Poisson process at rate $1$ and $\eta$ is an $\Exp(1)$ variable independent of the process $(\xi_i)_{i \geq 1}$ (see \cite[Example 38]{AldBan}).
By a theorem of R\"uschendorf (see \cite[Theorem 4.2]{R1}), under appropriate conditions on the random coefficients of the SFPEs above, there is a one-to-one relationship between the solutions of the max-type equation and the corresponding homogeneous equation. Therefore, it is convenient to study the homogeneous equation
\begin{equation}	\label{eq:Exa_hom}
X	~\stackrel{d}{=}~	\sup_{i \geq 1} A_i X_i,
\end{equation}
which is equivalent to Eq.\ \eqref{MinFP1} by an application of the involution $x \mapsto x^{-1}$. (Note that if in Eq.\ \eqref{eq:Exa_hom} $X$ has an atom in $0$, then this atom becomes an atom in $\infty$ in the equivalent equation \eqref{MinFP1}. This situation is not explicitly covered by the subsequent analysis but the results of this article remain true (after some minor changes) if an atom in $\infty$ is permitted.)

A first systematic approach to Eq.\ (\ref{MinFP1}) was given by Jagers and R\"osler \cite{JR} who pointed out
the following connection of \eqref{MinFP1} with its additive counterpart
\begin{equation} \label{SumFP1}
X\ \stackrel{d}{=}\ \sum_{i \geq 1} T_i X_i
\end{equation}
and the corresponding map $M_{\Sigma}:\mathcal{P}\rightarrow 
\overline{\mathcal{P}}$, defined by
\begin{equation}\label{MapST}
M_{\Sigma}(F)\  \gl\  \mathds{P} \left(\sum_{i \ge 1} T_{i}X_{i}\in \, \cdot \right), \quad X\stackrel{d}{=} F,
\end{equation}
which is usually called the \emph{smoothing transformation} due to Durrett and Liggett \cite{DL}.
Namely, rewriting \eqref{SumFP1} in terms of the Laplace transform $\varphi$, say, of $X$, we obtain
\begin{equation} \label{SumFP2}
\varphi(t)\ =\ \mathds{E} \prod_{i \geq 1} \varphi(t T_{i})
\end{equation}
for all $t \ge 0$, which is the direct analog of (\ref{MinFP2}). But since any Laplace transform vanishing at $\infty$ 
and thus pertaining to a distribution on $(0,\infty)$
can also be viewed as the survival function of a (continuous) probability distribution on $[0,\infty)$, one has the implication
\begin{equation}\label{FSigmainFinf}
\mathfrak{F}_{\Sigma} \not = \emptyset \quad\Longrightarrow\quad \mathfrak{F}_{\wedge} \not = \emptyset,
\end{equation}
where $\mathfrak{F}_{\Sigma}$ denotes the set of all positive solutions to \eqref{SumFP1}. Defining
$T^{(\alpha)}\gl (T_i^{\alpha})_{i\ge 1}$, one further has
\begin{equation} \label{FSigmaalphainFinf}
\mathfrak{F}_{\Sigma}(T^{(\alpha)}) \not = \emptyset \text{ for some } \alpha > 0 \quad\Longrightarrow\quad
\mathfrak{F}_{\wedge} \not = \emptyset,
\end{equation}
owing to the fact that $\mathfrak{F}_{\wedge} = \{\mathds{P}(X^{1/\alpha} \in \cdot): \, \mathds{P}(X\in\cdot) \in 
\mathfrak{F}_{\wedge}(T^{(\alpha)})\}$.
However, Jagers and R\"osler also give an example which shows that this implication cannot be reversed. We take up their example (the water cascades example) in Section \ref{cascade} in a generalized form.

Eq.\ \eqref{MinFP1} for the situation where the $T_{i},\,i\ge 1$, are deterministic but not necessarily nonnegative is discussed in detail by Alsmeyer and R\"osler \cite{AR2}. Their results concerning the case of nonnegative weights can be summarized as follows: Except for simple cases nontrivial fixed-points exist iff $T$ possesses a \emph{characteristic exponent}, defined as the unique positive number $\alpha > 0$ such that $\sum_{i \geq 1} T_i^{\alpha} = 1$. In this case the set of solutions $\mathfrak{F}_{\wedge}$ can be described as follows: For $\beta > 0$ and $r > 1$, let $\mathfrak{H}(r,\beta)$ be the set of left continuous, multiplicatively $r$-periodic functions $h:(0,\infty) \rightarrow (0,\infty)$ such that $t \mapsto h(t) t^{\beta}$ is nondecreasing. The distribution $F$ on $(0,\infty)$ with survival function
$$ \overline{F}(t) \ \gl\  e^{-h(t) t^{\beta}},   \quad t > 0 $$
is then called $r$-periodic Weibull distribution with parameters $h$ and $\beta$, in short $r$-$\Weib(h,\beta)$. Put $\mathfrak{W}(r,\beta) \gl \{r\perWeib(h,\beta): \, h \in \mathfrak{H}(r,\beta)\}$ for $\beta>0$ and $r>1$, and let $\mathfrak{W}(1,\beta) \gl \{\Weib(c,\beta): \, c > 0\}$ denote the set of ordinary Weibull distributions with parameter $\beta$, \textit{i.e.}, the set of distributions $F$ having a survival function $\overline{F}(t) = \exp(-ct^{\beta})$ ($t \ge 0$) for some positive constant $c$. Then $\mathfrak{F}_{\wedge} = \Weib(1,\alpha)$ or $\mathfrak{F}_{\wedge} = \mathfrak{W}(r,\alpha)$, respectively, depending on whether the closed multiplicative subgroup $\subseteq\mathds{R}^{+}= (0,\infty)$ generated by the positive $T_i$, which we denote by $\G(T)$, equals $\mathds{R}^{+}$ or $r^{\mathds{Z}}$ for some $r > 1$.
In view of this result, the result of Alsmeyer and R\"osler extends classical results on extreme value distributions and the problem adressed in this paper, the analysis of Eq.\ \eqref{MinFP1} is a further generalization of the analysis of extreme values, namely, of the distributional equation of homogeneous min-stability for in our situation the scaling factor is replaced by random coefficients.

One purpose of this paper is to investigate under which conditions similar results hold true in the situation of random weights $T_i$, $i \geq 1$, \textit{i.e.}, in which cases Weibull distributions or suitable mixtures of them are solutions to \eqref{MinFP1}. This calls for extended definitions
 of $\G(T)$ and of the characteristic exponent: We define $\G(T)$ as the minimal closed multiplicative subgroup $\G \subseteq\mathds{R}^{+}$ such that $\mathds{P}(T_i \in \G \cup \{0\} \text{ for all } i \in \mathds{N}) = 1$. We further define $m:[0,\infty) \mapsto [0,\infty]$ by
\begin{equation}\label{mgf}
m(\beta)\ \gl\  \mathds{E} \sum_{i \geq 1} T_i^{\beta}
\end{equation}
and then the characteristic exponent as the minimal positive $\alpha $ such that $m(\alpha) = 1$ if such an $\alpha$ exists. With these generalizations, we obtain a connection between certain Weibull mixtures and \emph{$\alpha$-regular solutions}, defined as solutions $F$ to \eqref{MinFP1} such that
the ratio $t^{-\alpha}(1-\overline{F}(t))$ stays bounded away from 0 and $\infty$ as $t$ approaches $0$. Indeed, Theorem \ref{alpha-elementar-Darstellung} will show that any $\alpha$-regular fixed point is a mixture of Weibull distributions with parameter $\alpha$ and particularly \emph{$\alpha$-elementary} which means that $t^{-\alpha}(1-\overline{F}(t))$
converges to a positive constant as $t\downarrow 0$ through a residue class
relative to $\G(T)$. Furthermore, Theorem \ref{alpha-elementar-Charakterisierung} will provide an exact characterization of when $\alpha$-elementary solutions to Eq.\ \eqref{MinFP1} exist.
Both, the existence of Weibull mixtures as fixed points and the existence of regular fixed points, are related to the existence of the characteristic exponent, which also plays a fundamental role in the analysis of Eq.\ \eqref{SumFP1}. The further organization of this article is as follows. Section \ref{Simple} provides a discussion of trivial and simple cases of \eqref{MinFP1}, which will be excluded thereafter. An introduction of the weighted branching model closely related to our SFPE \eqref{MinFP1} is given in Section \ref{WBP}, followed by the presentation
and discussion of the main results in Section \ref{MainResults}.
Section \ref{DisintReq} contains the derivation of a certain
pathwise renewal equation related to \eqref{MinFP1} via disintegration,
while Section \ref{CharakteristischerExponent} is devoted to a study of the characteristic exponent. It contains most of the necessary prerequisites
to prove Theorem \ref{alpha-elementar-Charakterisierung} and Theorem
\ref{alpha-elementar-Darstellung} which is done in Section \ref{FRE}. Here
the afore-mentionded pathwise renewal equation will form a key ingredient.
As already mentioned, Section \ref{cascade} provides a discussion of
Eq.\ \eqref{MinFP1} for a class of examples where the characteristic exponent does not generally exist. Finally, Section \ref{ConRem} contains some results for Eq.\ \eqref{SumFP1}, which are closely related to ours and can be derived by the same methods. Theorem \ref{alpha-elementar-Charakterisierung2} and Theorem \ref{alpha-elementar-Darstellung2} are extensions of Theorem 2 and Proposition 3 in \cite{Ik}.

\section{Basic results and simple cases} \label{Simple}

This section is devoted to a brief discussion of simple cases and a justification of the following two basic assumptions on $T$ (or, to be more precise, on the distribution of $T$): Put $N\gl\sum_{i\ge 1}
\mathbbm{1}_{\{T_{i}>0\}}$ and consider
\begin{equation} \tag{A1} \label{A1}
0\ <\ \mathds{P}(N>1)\ \leq\ \mathds{P}(N\ge 1)\ =\ 1;
\end{equation}
\begin{equation} \tag{A2} \label{A2}
\mathds{P}\bigg(\sup_{i \geq 1} T_i < 1\bigg) > 0.
\end{equation}
By our standing assumption, $\mathds{P}(N=0) = \mathds{P}(\sup_{i\ge 1}T_{i}=0)= 0$. Hence, if
\eqref{A1} fails to be true, then $N=1$ a.s.\ and Eq.\ \eqref{MinFP1} reduces to $X\stackrel{d}{=} T X$, where $T$ is independent of $X$ and a.s.\ positive. But this SFPE can easily be solved,
namely $\mathfrak{F}_{\wedge}\ne\emptyset$ iff $T=1$ a.s., see \textit{e.g.}\ Liu \cite{Liu}, Lemma 1.1. Validity  of condition \eqref{A1} will therefore always be assumed hereafter. As a consequence, the branching process with offspring distribution $\mathds{P}(N \in \, \cdot)$ (of simple Galton-Watson type if $N<\infty$ a.s.) survives with probability $1$, a fact that will be used later. 

The justification of assumption \eqref{A2} is slightly more involved and based upon the following two propositions: 

\begin{Prop} \label{beschrSatz1}
Suppose \eqref{A1}. Then 
$$\sup_{i \ge 1} T_i \geq 1\text{ a.s.}\quad\text{and}\quad\mathds{P}\left(\sup_{i \ge 1} T_i > 1
\right) >0 $$ 
implies $\mathfrak{F}_{\wedge} = \emptyset$.
\end{Prop}
\begin{proof}
Let $F$ be a solution to \eqref{MinFP1}. Then Eq.\ \eqref{MinFP2} gives
$$ \overline{F}(t)\ =\ \mathds{E} \prod_{i \geq 1} \overline{F}(t T_i)\ 
      \le\ \mathds{E} \overline{F}\left(t  \sup_{i \geq 1} T_i\right)\ 
          \le\ \overline{F}(t) $$
and thus $\overline{F}(t) = \mathds{E} \overline{F}(t  \sup_{i \ge 1} T_i)$ for all $t \ge 0$. 
Let $(Y_i)_{i \geq 1}$ be a sequence of i.i.d.\ copies of $\sup_{i \geq 1} T_i$ with associated multiplicative random walk $(\Pi_{n})_{n\ge 0}$, \textit{i.e.}, $\Pi_{0}\gl 1$ and $\Pi_{n}\gl Y_{1}\cdot\ldots\cdot
Y_{n}$ for $n\ge 1$. Then $\overline{F}(t) = \mathds{E} \overline{F}(t \Pi_n)$ for each $n$. Our assumptions on $\sup_{i \ge 1} T_i$ ensure $\Pi_n \uparrow \infty$ a.s., whence $\overline{F}(t) = 0$ for all $t > 0$, that is $F=\delta_{0}$.
\end{proof}

Before proceeding with our second proposition, let us note in passing that any $\sigma(T)$-measurable finite or infinite rearrangement $T_{\pi}\gl (T_{\pi(1)},T_{\pi(2)},\ldots)$ of $T$ leaves the set of solutions
to our SFPE \eqref{MinFP1} unchanged because the $X_{i}$ are i.i.d.\ and independent of $T$,
thus also of $T_{\pi}$. So $\mathfrak{F}_{\wedge}(T)=\mathfrak{F}_{\wedge}(T_{\pi})$. As a consequence
it is no loss of generality to assume $T_{1}=\sup_{i\ge 1}T_{i}$ whenever the supremum is a.s.\
attained.

\begin{Prop} \label{beschrSatz2}
If \eqref{A1} holds the following assertions are equivalent:
\begin{itemize}
 \item[(a)]
  $\sup_{i \geq 1} T_i = 1$ a.s.
 \item[(b)]
  $\exists\,0<\gamma\le 1: \mathfrak{F}_{\wedge} =
    \{F \in \mathcal{P}:F([\gamma c,c])=1\text{ for some }c>0\}$.
 \item[(c)]
  $\delta_c \in \mathfrak{F}_{\wedge}$ for all $c>0$.
 \item[(d)]
  $\delta_c \in \mathfrak{F}_{\wedge}$ for some $c>0$.
\end{itemize}
\end{Prop}
\begin{proof}
The implications "(b) $\Rightarrow$ (c) $\Rightarrow$ (d) $\Rightarrow$ (a)" being obvious, we must only 
prove "(a) $\Rightarrow$ (b)". Define $\gamma\gl 1$,
if $\mathds{P}(\exists\, i \in \mathds{N}: \, T_i = 1) < 1$, and $\gamma\gl\esssup_{i\ge 2}T_{i}$,
if  $\mathds{P}(\exists\, i \in \mathds{N}: \, T_i = 1) = 1$ and w.l.o.g.\ $T_1 = \sup_{i \geq 1} T_i$. 
Note for the latter situation that $\gamma>0$ by \eqref{A1}.
The two cases $0<\gamma<1$ (Case 1) and $\gamma=1$ (Case 2) will now be discussed separately.

\vspace{.2cm}\noindent
{\bf Case 1.} Pick any $F\in \mathfrak{F}_{\wedge}$ and let $X, X_1,X_2,\ldots$ denote a sequence of i.i.d.\ random variables with common distribution $F$. As $T_1 = 1$ a.s.\ the SFPE reads
\begin{equation} \label{VtlgUngl}
X\ \stackrel{d}{=}\ X_1 \wedge \inf_{i \geq 2} \frac{X_i}{T_i}.
\end{equation}
and clearly entails $X_1 \leq \inf_{i \geq 2} X_i/T_i$ a.s. Now use the independence of $X_{1}$
and $\inf_{i \geq 2} X_i/T_i$ to infer that this can only hold if
$$ X_{1}\ \le\ c\ \le\ \inf_{i \geq 2} X_i/T_i\quad\text{a.s.} $$
for some $c$, w.l.o.g.\ $c\gl\esssup X_{1}$ which is positive, for $F\ne\delta_{0}$.
It remains to show that $X \ge \gamma c$ a.s. Assuming the contrary, we also have $\mathds{P}(X < u \gamma c) > 0$ for some $u \in (0,1)$. Since $\gamma=\esssup_{i\ge 2}T_{i}$ in the present case, the stopping time $\nu\gl\inf\{i\ge 2:T_{i}>u\gamma\}$ is finite with positive probability, and we have that $\mathds{P}(X_{\nu}\in\cdot|\nu<\infty)=\mathds{P}(X\in\cdot)$ because the $X_{i}$ and
$T$ are independent. But this leads to the following contradiction:
\begin{eqnarray*}
0
& = &
\mathds{P} \left( \inf_{i \geq 2} \frac{X_i}{T_i} < c \right) \\
& \geq &
\mathds{P} \left(\nu<\infty,X_{\nu}<u\gamma c \right) \\
& \geq &
\mathds{P}(\nu<\infty)\,\mathds{P} \left( X < u \gamma c \right)\ >\ 0.
\end{eqnarray*}
So we have proved $F([\gamma c,c])=1$. Conversely, if we pick any $F\in \mathcal{P}$ with
this property for some $c>0$, then $F\in \mathfrak{F}_{\wedge}$ is immediate from \eqref{VtlgUngl} because we have there $X_1 \wedge \inf_{i \geq 2} X_i/T_i=X_{1}$ a.s.

\vspace{.2cm}\noindent
{\bf Case 2.} If $\gamma=1$ we cannot assume $T_{1}=\sup_{i\ge 1}T_{i}$ and then resort to the above argument because with positive probability the supremum may not be attained. On the other hand, the claim reduces
here to $\mathfrak{F}_{\wedge}=\{\delta_{c}:c>0\}$ and it is easily verified that any $\delta_{c}$
is indeed a solution. For the reverse inclusion, pick any solution $F$ and suppose it is not concentrated at a single point, thus $\overline{F}(s)\in (0,1)$ for some $s>0$. 
For $r \in (0,1)$, define a r.v. $U_r$ as follows:
$$ U_r\ \gl\ 
\begin{cases}
\sup_{i \not = k} T_i             & \text{if there is a $k \geq 1$ such that $T_k = 1$,}        \\
T_{\tau_r}                & \text{if $T_i < 1$ for all $i \geq 1$,}
\end{cases} $$
where $\tau_r \gl \inf\{i \geq 1: r < T_i < 1\}$. Observe that in the second case $N=\infty$ and 
$U_{r_{k}}\uparrow 1$ for any choice $r_{k}\uparrow 1$, so that $\prod_{i\ne\tau_{r}}\overline{F}
(tT_{i})\le\lim_{k\to\infty}\overline{F}(tU_{r_{k}})=\overline{F}(t)$ by left continuity. 
For any $t \ge 0$, we now infer
$$ \overline{F}(t)
 \ =\ \mathds{E} \prod_{i \geq 1} \overline{F}(t T_i)
 \ \le\ \overline{F}(t) \cdot \mathds{E} \overline{F}(t U_r) $$
and therefore $\mathds{E} \overline{F}(t U_r)=1$ for any $t$ such that $\overline{F}(t)\in (0,1)$. 
By left continuity, $\overline{F}(rt) < 1$ for any such $t$ and some $r\in (0,1)$. However, $\mathds{P}
(U_r > r)>0$ then leads to the contradiction
$$ 1\ =\ \mathds{E} \overline{F}(t U_r)\ \le\ \mathds{P}(U_r \leq r) + \overline{F}(rt)\,
\mathds{P}(U_r > r)\  <\ 1. $$
We hence conclude that $F$ must be concentrated in a single point.
\end{proof}

\begin{Rem} \label{CompactSupport}
\rm As a particular consequence of Proposition  \ref{beschrSatz2},  all solutions to \eqref{MinFP1} have compact support if \eqref{A1} and $\sup_{i \geq 1} T_i = 1$ a.s.\  hold true. As to a reverse conclusion, let us point out the following:

\vspace{.2cm}\noindent
If \eqref{A1} holds true and $\mathbb{P}(N<\infty)=1$, then the assertions
\begin{itemize}
 \item[(a)]
  $\sup_{i \geq 1} T_i = 1$ a.s.
 \item[(b)]
  There exists $F \in \mathfrak{F}_{\wedge}$ with compact support.
\end{itemize}
are equivalent.

\vspace{.2cm}
With only "(b) $\Rightarrow$ (a)" to be proved, let $F$ be an element of $\mathfrak{F}_{\wedge}$ with compact support and $X$ a random variable with distribution $F$, so $C\gl \esssup X\in (0,\infty)$. Suppose now there exists $r\in (0,1)$ such that $q\gl\mathds{P}(\sup_{i \geq 1} T_i \leq r) > 0$. We can pick $r$ and $t>C$ in such a way that
$rt<C<t$ and thus $\overline{F}(rt)>\overline{F}(t)=0$. But then
\begin{eqnarray*}
0\ =\ \overline{F}(t)& = &
\mathds{E} \prod_{i \geq 1} \overline{F}(t T_i) \\
& \geq &
\mathds{E} \prod_{i \geq 1} \overline{F}(t T_i)\, \mathbbm{1}_{\{\sup T_i \leq r \}}  \\
& \geq & \mathds{E}\overline{F}(rt)^N \mathbbm{1}_{\{\sup T_i \leq r \}} \ >\ 0.
\end{eqnarray*}
which is a contradiction. Consequently, $\mathds{P}(\sup_{i \geq 1} T_i \geq 1) = 1$ which in combination with $\mathfrak{F}_{\wedge} \not = \emptyset$ and Proposition \ref{beschrSatz1}
proves (a).
\end{Rem}

We close this section with a lemma that shows that any $F\in\mathfrak{F}_{\wedge}$
is continuous at 0 and that a search for solutions putting mass on $[0,\infty)$ is
actually no restriction.

\begin{Lemma}  \label{wlog}
Suppose \eqref{A1} and let $F\ne\delta_{0}$ be any distribution on $\mathds{R}$ solving Eq.\ \eqref{MinFP1}. Then $F$ is continuous at $0$ and concentrated either on $[0,\infty)$ or $(-\infty,0]$. In the latter case, if $X\stackrel{d}{=}F$ and
$G$ denotes the distribution of $-X^{-1}$ ($<\infty$ a.s.), then $G\in\mathfrak{F}_{\wedge}
(T^{-1})$, where $T^{-1}\gl (T_{i}^{-1}\mathds{1}_{\{T_{i}>0\}})_{i\ge 1}$.
\end{Lemma}

\begin{proof}
Let $X,X_{1},X_{2},\ldots$ be i.i.d.\ with distribution $F$ and independent of $T$. 
In view of what has been mentioned before Proposition \ref{beschrSatz2} we may assume w.l.o.g. that $T_{i}>0$ iff $N\ge i$. Then
\begin{eqnarray*}
\overline{F}(0)
&=&
\mathds{P}(X\ge 0)\\
&=&
\mathds{P}(X_{i}\ge 0\text{ for }1\le i\le N)
\ =\ \mathds{E}\overline{F}(0)^{N},
\end{eqnarray*}
whence $\overline{F}(0)$ must be a fixed point of the generating function of $N$
in $[0,1]$. Now use \eqref{A1} to infer $\overline{F}(0)\in\{0,1\}$. Next 
consider $\mathds{P}(X>0)=\overline{F}(0+)$ and suppose it to be $<1$. Then we infer with the help of \eqref{MinFP1}
\begin{eqnarray*}
\overline{F}(0+)
&=&
\lim_{t\downarrow 0}\overline{F}(t)\\
&\le&
\lim_{t\downarrow 0}\mathds{E}\prod_{i=1}^{N\wedge n}\overline{F}(t
T_{i})\ =\ \mathds{E}\overline{F}(0+)^{N\wedge n}
\end{eqnarray*}
for each $n\ge 1$ and thereupon $\overline{F}(0+)\le\mathds{E}\overline{F}(0+)^{N}\mathds{1}_{\{N<\infty\}}$. On the other hand, by another appeal to \eqref{A1}, $\mathds{E}s^{N}\le s$ for each $s\in [0,1)$ with equality holding iff $s=0$.
Consequently, $\overline{F}(0+)=0$, which is clearly impossible as $F\ne\delta_{0}$. We thus conclude $\overline{F}(0)=\overline{F}(0+)=1$
and thereby $\mathds{P}(X=0)=\overline{F}(0)-\overline{F}(0+)=0$ which proves the continuity of $F$ at 0.
As for the final assertion, it suffices to note that $X\stackrel{d}{=}\inf_{i\ge 1}X_{i}/T_{i}$ is clearly equivalent to $-X^{-1}\stackrel{d}{=}\inf_{i\ge 1}(-X_{i}^{-1})T_{i}$.
\end{proof}

Unless stated otherwise, we will always assume \eqref{A1} and \eqref{A2} hereafter. As a consequence of \eqref{A2}, 
we infer that the closed multiplicative subgroup $\G(T)$ generated by $T$ cannot be $\{1\}$, the trivial subgroup. So we have either $\G(T) = r^{\mathds{Z}}$ for some $r > 1$ ($r$-geometric case) or $\G(T) =\mathds{R}^{+}$ (continuous case).

\section{Connection with weighted branching processes} \label{WBP}

Let $\V$ be the infinite tree with vertex set $\bigcup_{n \in \mathds{N}_0} \mathds{N}^n$, where $\mathds{N}^0$ contains only the empty tuple $\varnothing$. We abbreviate $v = (v_1,\ldots,v_n)$ by $v_1 \ldots v_n$ and write $vw$ for the vertex $(v_1,\ldots,v_n,w_1,\ldots,w_m)$, where $w = (w_1,\ldots,w_m)$. Furthermore, $|v|=n$ and $|v|<n$ will serve as shorthand notation for $v \in \mathds{N}^n$ and $v \in \mathds{N}^k$ for some $k < n$, respectively. $|v| \leq n$, $|v| \geq n$ and $|v| > n$ are defined similarly. Let $(T(v))_{v \in \V}$ denote a family of i.i.d.\ copies of $T$. For sake of brevity, suppose $T=T(\varnothing)$. 
Interpret $T_i(v)$ as a weight attached to the edge $(v,vi)$ in the infinite tree $\V$. Then put $L(\varnothing) \gl 1$ and
$$ L(v)\ \gl\ T_{v_1}(\varnothing) \cdot \ldots \cdot T_{v_n}(v_1 \ldots v_{n-1}) $$
for $v = v_1 \ldots v_n \in \V$. So $L(v)$ gives the total multiplicative weight along the unique path  from $\varnothing$ to $v$. For $n \geq 1$, let $\mathcal{A}_n$ denote the $\sigma$-algebra generated by the sequences $T(v)$, $|v|<n$, \textit{i.e.},
$$ \mathcal{A}_n\ \gl\  \sigma \left(T(v): |v| < n \right) $$
Put  $\mathcal{A}_0\gl \{\emptyset, \Omega\}$ and $\mathcal{A}_{\infty}\gl \sigma(\mathcal{A}_n: \, n \geq 0)=\sigma(T(v):v\in\V)$.

Let us further introduce the following bracket operator $[\cdot]_u$ for any $u \in \V$. Given any
function $\Psi=\psi((T(v))_{v \in \V})$ of the weight ensemble $(T(v))_{v \in \V}$ pertaining to $\V$, define 
$$ [\Psi]_u\ \gl\ \psi((T(uv))_{v \in \V}) $$
to be the very same function, but for the weight ensemble pertaining to the subtree rooted at $u$.
Any branch weight $L(v)$ can be viewed as such a function, and we then obtain
$[L(v)]_u=T_{v_1}(u) \cdot \ldots \cdot T_{v_n}(uv_1 \ldots v_{n-1})$ if $v = v_1 \ldots v_n$, and thus
$[L(v)]_u = L(uv)/L(u)$ whenever $L(u)>0$.

The weighted branching process (WBP) associated with $(T(v))_{v \in \V}$ is now defined as
$$ W_n\ \gl\  \sum_{|v|=n} L(v), \quad n \geq 0.  $$
For any $\alpha \geq 0$, we can replace the $T(v)$ with $T^{(\alpha)}(v)\gl(T_{i}(v)^{\alpha})_{i\ge 1}$ which leads to the branch weights $L^{(\alpha)}(v)\gl L(v)^{\alpha}$ and the associated
WBP 
$$ W_n^{(\alpha)}\ \gl\ \sum_{|v|=n} L(v)^{\alpha},\quad n\ge 0. $$
Note that $T^{(0)}(v)=(\mathbbm{1}_{\{T_{i}(v)>0\}})_{i\ge 1}$, so that $W_{n}^{(0)}=\sum_{|v|=n}
\mathbbm{1}_{\{L(v)>0\}}$ counts the positive branch weights in generation $n$. If $N<\infty$ a.s., then $(W_{n}^{(0)})_{n\ge 0}$ forms a
Galton-Watson process with offspring distribution $\mathds{P}(N\in\cdot)$, for
$W_{1}^{(0)}\stackrel{d}{=}N$. Suppose there exists an $\alpha > 0$ such that $m(\alpha) \leq 1$
with $m$ as defined in \eqref{mgf}. Then the sequence $(W_n^{(\alpha)})_{n \geq 0}$ constitutes a nonnegative supermartingale with respect to $(\mathcal{A}_n)_{n \geq 0}$ and hence converges a.s.
to $W^{(\alpha)}\gl \liminf_{n \to \infty} W_n^{(\alpha)}$. By Fatou's lemma,
$$ 0\ \le\ \mathds{E} W^{(\alpha)}\ \leq\ \liminf_{n \to \infty} \mathds{E} W_n^{(\alpha)}
\ =\ \lim_{n\to\infty}m(\alpha)^{n}\ \leq\ 1. $$
which gives $W^{(\alpha)} = 0$ a.s. if $m(\alpha) < 1$. In the case $m(\alpha)=1$, we have the dichotomy $\mathds{E} W^{(\alpha)} = 0$ or $\mathds{E} W^{(\alpha)} = 1$ (\textit{cf.}\ Theorem \ref{Biggins} in the Appendix, or Biggins \cite{Bi}, Lyons \cite{Ly} and Alsmeyer and Iksanov \cite{AI} for details). Henceforth, let $\Lambda_{\alpha}$ and $\varphi_{\alpha}$ denote the distribution and Laplace transform, respectively, of $W^{(\alpha)}$. 

\begin{Rem} \label{remcharexp} \rm
As $m(\alpha)=1$ and $\mathds{E} W^{(\alpha)}=1$ (or, equivalently, $\mathds{P}(W^{(\alpha)}>0)>0$) for some $\alpha>0$ will be a frequent assumption hereafter, it is noteworthy that this forces $\alpha$ to be the characteristic exponent of $T$, that is, the \emph{minimal} $\beta>0$ with $m(\beta)=1$. For a proof using Theorem \ref{Biggins} we refer to Corollary \ref{charcharexp} in the Appendix. Due to our standing assumption \eqref{A1}, Theorem \ref{Biggins}
further implies that $\mathds{P}(W^{(\alpha)}>0)>0$ is actually equivalent to the a.s.\
positivity of $W^{(\alpha)}$.
\end{Rem}

As explained before Proposition \ref{beschrSatz2}, the mapping $M$ defined in \eqref{MapFP} is invariant under $\sigma(T)$-measurable rearrangements of the $T_i$, $i \geq 1$. It is therefore stipulated hereafter that $T_i > 0$ if, and only if, $1\le i \le N$.

In order to provide the connection of the previously introduced weighted bran\-ching model with the SFPE \eqref{MinFP1}, let $(X(v))_{v\in\V}$ be a family of independent copies of $X$ which is also independent of $(T(v))_{v\in\V}$. If $X\stackrel{d}{=}F$, then $n$-fold iteration of
\eqref{MinFP1} yields
\begin{equation} \label{MinFPWB}
X \stackrel{d}{=} \inf_{|v|=n} \frac{X(v)}{L(v)}
\end{equation}
for all $n \geq 0$ which in terms of the survival function $\overline{F}$ becomes
\begin{equation} \label{MinFPWBsf}
\overline{F}(t) = \mathds{E} \prod_{|v|=n} \overline{F}(t L(v)),\quad t \ge 0.
\end{equation}

\section{Main results}\label{MainResults}

We continue with the statement of the two main results that will be derived in this article. Theorem \ref{alpha-elementar-Charakterisierung} provides the connection between the existence of certain regular solutions to \eqref{MinFP1} and the existence of the characteristic exponent of $T$, while Theorem \ref{alpha-elementar-Darstellung} is a representation result which states that any regular solution is a certain Weibull mixture (cf.\ Definition \ref{Wrnualpha} below). The definition of an $\alpha$-regular fixed point is part of the following definition.

\begin{Def} \label{fixptype} \rm
Let $\alpha>0$ and $F\in\mathfrak{F}_{\wedge}$. Put $D_{\alpha}\overline{F}(t)\gl 
t^{-\alpha}(1-\overline{F}(t))$ for $t>0$. Then $F$ is called
\begin{itemize}
\item[{(1)}] \emph{$\alpha$-bounded}, if $\limsup_{t\downarrow 0}D_{\alpha}\overline {F}(t) <\infty$.
\item[{(2)}] \emph{$\alpha$-regular}, if 
  $0<\liminf_{t \downarrow 0}D_{\alpha}\overline{F}(t) \le \limsup_{t     \downarrow 0} D_{\alpha}\overline{F}(t) <\ \infty$.
\item[{(3)}] \emph{$\alpha$-elementary}, if
 \begin{itemize}
 \item[{---}] in the case $\G(T)=\mathds{R}^{+}$ --- there exists a    constant $c > 0$ such that $\lim_{t\downarrow 0}D_{\alpha}\overline{F}(t)=c$.
 \item[{---}] in the case $\G(T) = r^{\mathds{Z}}$ for some $r > 1$ --- for each   $s \in [1,r)$ there exists a positive constant $h(s)$ such that 
  $\lim_{n\to\infty}D_{\alpha}\overline{F}(sr^{-n})=h(s)$ for each 
  $s\in [1,r)$.
 \end{itemize}
\end{itemize}
The sets of $\alpha$-bounded, $\alpha$-regular and $\alpha$-elementary fixed point are denoted by $\mathfrak{F}_{\wedge,b}^{\alpha},\mathfrak{F}_{\wedge,r}^{\alpha}$ and $\mathfrak{F}_{\wedge,e}^{\alpha}$, respectively.
\end{Def}

\begin{Rem} \rm
(a) The notion of an \emph{$\alpha$-elementary} fixed point has been introduced by Iksanov \cite{Ik} in his study of the smoothing transformation $M_{\Sigma}$
given in \eqref{MapST} and the associated SFPE \eqref{SumFP1}. His definition is the same as ours for the continuous case when replacing $D_{\alpha}\overline{F}$ with $D_{\alpha}\varphi$ where $\varphi$ denotes the Laplace transform of a solution to \eqref{SumFP1}. 

\vspace{.2cm}\noindent
(b) Given the existence of the characteristic exponent $\alpha$, Guivarc'h \cite{Gui} and later Liu \cite{Liu} called a (nonnegative) solution to \eqref{SumFP1} with Laplace transform $\varphi$ \emph{canonical} if it can be obtained as the stable transformation of a solution to the very same equation for the weight vector $T^{(\alpha)}=(T_{i}^{\alpha})_{i\ge 1}$. If the latter solution has Laplace transform $\psi$ this means that $\varphi(t)=\psi(t^{\alpha})$ for all $t\ge 0$. This definition appears to be more restrictive than that of an $\alpha$-elementary fixed point because the latter definition is valid for any $\alpha>0$. On the other hand, once shown that an $\alpha$-elementary fixed point actually exists only if $\alpha$ is the characteristic exponent of $T$ (see Theorem \ref{alpha-elementar-Charakterisierung}), "$\alpha$-elementary" (at least in the more restrictive sense of Iksanov) and "canonical" turn out to be just different names for the same objects (see Theorem 2 in \cite{Ik} and also Theorem \ref{alpha-elementar-Darstellung} below).

\vspace{.2cm}\noindent
(c) Note for the $r$-geometric case that $\inf_{s\in(1,r]}h(s)$ must be positive,
for $(sr^{n})^{-\alpha}(1-\overline{F}(sr^{n}))\ge (s/r)^{-\alpha}r^{-\alpha(n+1)}(1-\overline{F}(r^{n+1}))$ for all $s\in (1,r]$ and $n\in\mathds{Z}$. After this observation, we see that any $\alpha$-elementary fixed point is also $\alpha$-regular, and since $\alpha$-regularity trivially implies $\alpha$-boundedness, we have that
$$ \mathfrak{F}_{\wedge,e}^{\alpha}\ \subseteq\ \mathfrak{F}_{\wedge,r}^{\alpha}\ \subseteq\ \mathfrak{F}_{\wedge,b}^{\alpha}. $$
\end{Rem}

\begin{Theorem} \label{alpha-elementar-Charakterisierung}
Suppose \eqref{A1} and \eqref{A2}. Then the following assertions are equivalent for any $\alpha>0$:
\begin{itemize}
 \item[(a)]  Eq.\ \eqref{MinFP1} has an $\alpha$-elementary      solution
    $(\mathfrak{F}_{\wedge,e}^{\alpha}\ne\emptyset)$.
 \item[(b)]  Eq.\ \eqref{MinFP1} has an $\alpha$-regular solution
    $(\mathfrak{F}_{\wedge,r}^{\alpha}\ne\emptyset)$.

 \item[(c)]  $m(\alpha)=1$ and $\mathds{P}(W^{(\alpha)}>0)>0$.
 \item[(d)] $m(\alpha)=1$, the random walk $(\overline{S}_{\alpha,n})_{n \ge 0}$ with increment distribution $\Sigma_{\alpha,1}(B)\gl\mathds{E}\sum_{i\ge 1}T_{i}^{\alpha}\mathbbm{1}_{B}(T_{i})$, $B\in\mathfrak{B}$, converges to $\infty$ a.s.\ and
    $$ \int_{(1,\infty)} \left[\frac{u \log u}
    {\mathds{E}(\overline{S}_{\alpha,1}^+ \wedge \log u)} \right] 
    \mathds{P} (W_1^{(\alpha)} \in \, du) < \infty. $$
\end{itemize}
\end{Theorem}

The proof of this theorem will be given in Section \ref{FRE}. We note that the equivalence statement of (c) and (d) is part of Theorem \ref{Biggins} and only included here for completeness. Let us further point out the connection of our result with a similar one on $\alpha$-elementary fixed points of the smoothing transformation obtained by Iksanov \cite{Ik}. By Lemma A.3 in \cite{Ik}, each \emph{continuous} $\alpha$-elementary solution $\overline{F}$  to \eqref{MinFP2} is the Laplace transform of a probability measure on $[0,\infty)$ solving \eqref{SumFP1}. Therefore, under the continuity restriction, a part of our theorem could be deduced from Theorem 2 in \cite{Ik}. On the other hand, the latter result strongly hinges on Proposition 1 in the same reference the proof of which contains a serious flaw (occuring in Eq.\ (14) on p.\,36 where it is mistakenly assumed that $q$ does not depend on $v$. Without this assumption the subsequent argument breaks down completely and there seems to be no obvious way to fix it under the stated assumptions).

In order to state our second theorem, the following definition of certain classes of Weibull mixtures is given, where the definitions of $r\perWeib(h,\alpha)$, $\Weib(c,\alpha)$ and $\mathfrak{H}(r,\alpha)$ should be recalled from the Introduction.

\begin{Def} \label{Wrnualpha} \rm
Let $\alpha > 0$ and $\Lambda$ be a probability measure on $\mathds{R}^{+}$.
Then
\begin{itemize}
 \item[(a)]
 $\mathfrak{W}_{\Lambda}(1,\alpha)$ denotes the class of 
 $\Lambda$-mixtures of $\Weib(c,\alpha)$ distributions $F$ of the form
 $$ F(\cdot) = \int \Weib(yc,\alpha)(\cdot) \, \Lambda(dy),  $$
  where $c > 0$.
 \item[(b)]
 $\mathfrak{W}_{\Lambda}(r,\alpha)$ for $r>1$ denotes the class of    $\Lambda$-mixtures of $r\perWeib(h,\alpha)$ distributions $F$ of the form
 $$ F(\cdot) = \int r\perWeib(yh,\alpha)(\cdot) \, \Lambda(dy), $$
 where $h \in \mathfrak{H}(r,\alpha)$.
\end{itemize}
\end{Def}

The reader should notice that $\mathfrak{W}_{\Lambda}(1,\alpha)$ is always a
subclass of $\mathfrak{W}_{\Lambda}(r,\alpha)$ for any $r>1$.

\begin{Theorem} \label{alpha-elementar-Darstellung}
Suppose \eqref{A1} and that $m(\alpha)=1$ and $\mathds{E}W^{(\alpha)}=1$ for some $\alpha>0$. 
Recall that $\Lambda_{\alpha}=\mathds{P}(W^{(\alpha)}\in\cdot)$. Then
\begin{equation*}
 \mathfrak{F}_{\wedge,b}^{\alpha} = \mathfrak{F}_{\wedge,r}^{\alpha}  = \mathfrak{F}_{\wedge,e}^{\alpha} = \mathfrak{W}_{\Lambda_{\alpha}}  (d,\alpha),
\end{equation*}
where $d=r>1$ in the $r$-geome\-tric case $(\G(T)=r^{\mathds{Z}})$
and $d=1$ in the continuous case $(\G(T)=\mathds{R}^{+})$.
\end{Theorem}

As to the proof of Theorem \ref{alpha-elementar-Darstellung}, let us note that,
once $m(\alpha)=1$ and $\mathds{E}W^{(\alpha)}=1$ have been verified,
the inclusion $ \mathfrak{W}_{\Lambda_{\alpha}}(d,\alpha)\subseteq
\mathfrak{F}_{\wedge,b}^{\alpha}$ follows upon direct inspection relying on the well known fact that $\Lambda_{\alpha}\in\mathfrak{F}_{\Sigma}(T^{(\alpha)})$, see Lemma \ref{WrnualphainFinf}.
So the nontrivial part is the reverse conclusion which will be shown in Section \ref{FRE}.

The reader should further notice that \eqref{A2} does not need to be assumed in Theorem \ref{alpha-elementar-Darstellung} because it already follows from \eqref{A1} and $m(\alpha)=1$.

\begin{Rem} \label{commentsthm}\rm
(a) The two previous theorems can be summarized as follows: The existence
of at least one $\alpha$-regular fixed point is equivalent to $\alpha$ being
the characteristic exponent of $T$ with $\mathds{E}W^{(\alpha)}=1$, and in this case \emph{all} regular solutions are in fact Weibull mixtures with
mixing distribution $\Lambda_{\alpha}$ and particularly $\alpha$-elementary.
Moreover, there are no further solutions in $\mathfrak{F}_{\wedge,b}^{\alpha} \backslash\mathfrak{F}_{\wedge,r}^{\alpha}$.

\vspace{.2cm}\noindent
(b) Let us briefly address two natural questions that arise in connection with our results. First, do further nontrivial solutions to \eqref{MinFP1} exist if
$m(\alpha)=1$ and $\mathds{E}W^{(\alpha)}=1$ for some $\alpha>0$? Clearly, any further $F\in\mathfrak{F}_{\wedge}$ must satisfy either
$$ \lim_{t\downarrow 0}D_{\alpha}\overline{F}(t)\ =\ \infty $$
or
$$  0\ \le\ \liminf_{t\downarrow 0}D_{\alpha}\overline{F}(t)\ <
\ \limsup_{t\downarrow 0}D_{\alpha}\overline{F}(t)\ =\ \infty. $$
Lemma \ref{liminflimsup} will show that only the
second alternative ($D_{\alpha}\overline{F}$ oscillating at 0) might be
possible. However, whether solutions of that kind really exist in certain instances remains an open question.

Second, one may wonder about the existence of solutions to \eqref{MinFP1} if $T$ does not possess a characteristic exponent. Although we cannot provide a general answer to this question, it will emerge from our discussion in Section \ref{cascade} that there are situations in which there is no characteristic exponent and yet $\mathfrak{F}_{\wedge}\ne\emptyset$. This was already observed by Jagers and R\"osler \cite{JR}, and the class of examples studied here forms a natural extension of theirs.

\vspace{.2cm}\noindent
(c) There is yet another situation, called the \emph{boundary case} by Biggins and Kyprianou \cite{BK3}, that we have deliberately excluded here from our analysis in order to not overburden the subsequent analysis.
It occurs when $m(\alpha)=1$ and $\mathfrak{F}_{\Sigma}(T^{(\alpha)})$ contains an element $\Lambda_{\alpha}^{*}$, say, with infinite mean for some $\alpha>0$. Then $W^{(\alpha)}=0$ a.s.\ and $m'(\alpha)
=0$ provided that $m(\cdot)$ exists in a neighborhood of $\alpha$.
This case is quite different from the one in focus here, where $\mathds
{E}W^{(\alpha)}=1$, except that $\mathfrak{W}_{\Lambda_{\alpha}^{*}}(d,\alpha)\subset\mathfrak{F}_{\wedge}$ with $d$ as in Theorem \ref{alpha-elementar-Darstellung} is easily verified by copying the proof of
Lemma \ref{WrnualphainFinf}.
If $\varphi_{\alpha}^{*}$ denotes the Laplace transform of $\Lambda_{\alpha}^{*}$, then, under mild conditions (cf.\ Theorem 5 of \cite{BK3}), $1-\varphi_{\alpha}^{*}(t)$ behaves like a constant times $t|\log t|$ as $t\downarrow 0$, and thus $\lim_{t\downarrow 0}D_{\alpha}\overline{F}(t)=\infty$ for any
$F\in\mathfrak{W}_{\Lambda_{\alpha}^{*}}(d,\alpha)$. We quote this different behavior as opposed to that in the situation of the results above to argue that the boundary case requires separate treatment. We refrain from going into further details and refer to a future publication.
\end{Rem}

\section{Disintegration and a pathwise renewal equation}\label{DisintReq}

Our further analysis is based on a disintegration of Eq.\ \eqref{MinFPWBsf} by which we mean
the derivation of a pathwise counterpart (Eq.\ \eqref{DisintegratedFPE} below) which reproduces
\eqref{MinFPWBsf} upon integration on both sides. We embark on the following known result on the sequence
\begin{equation} \label{multimart}
\overline{\mathcal{F}}_n(t)\ \gl\ \prod_{|v|=n} \overline{F}(t L(v)) ,\quad n \ge 0
\end{equation}
appearing under the expected value in \eqref{MinFPWBsf}.

\begin{Lemma} \label{Disintegration}
Let $F\in \mathfrak{F}_{\wedge}$. Then $(\overline{\mathcal{F}}_n(t))_{n \geq 0}$ forms a bounded nonnegative martingale with respect to $(\mathcal{F}_n)_{n \geq 0}$ and thus converges a.s. and in mean to a random variable $\overline{\mathcal{F}}(t)$ satisfying
$$ \mathds{E} \overline{\mathcal{F}}(t) = \overline{F}(t). $$
\end{Lemma}
\begin{proof}
The proof of this lemma can be found in Biggins and Kyprianou (\cite{BK1}, Theorem 3.1).
\end{proof}

In the situation of Lemma \ref{Disintegration}, we put
$$ \mathcal{F}(t)\ \gl\ 1-\liminf_{n \to \infty} \overline{\mathcal{F}}_{n}(t).  $$
and call the stochastic process $\mathcal{F}=(\mathcal{F}(t))_{t \geq 0}$ a {\it disintegration of $F$} and also a {\it disintegrated fixed point}.
The announced pathwise fixed-point equation for an arbitrary disintegrated fixed point is next.

\begin{Lemma} \label{disintSFPE}
Let $F\in \mathfrak{F}_{\wedge}$ and $\mathcal{F}$ a disintegration of $F$. Then
\begin{equation} \label{DisintegratedFPE}
\overline{\mathcal{F}}(t)\ =\ \prod_{|v|=n} [\overline{\mathcal{F}}]_v(t L(v))\quad\text{ a.s.}
\end{equation}
for each $t \geq 0$ and $n \in \mathds{N}_0$.
\end{Lemma}
\begin{proof}
We have
\begin{eqnarray*}
\overline{\mathcal{F}}(t)
& = &
\liminf_{k \to \infty} \prod_{|v|=n+k} \overline{F}(t  L(v)) \\
& = &
\liminf_{k \to \infty} \prod_{|v|=n} \prod_{|w|=k} \overline{F}(t  L(v)  [L(w)]_v) \\
& \leq &
\liminf_{k \to \infty} \prod_{v \in \{1,\ldots,m\}^n} [\overline{\mathcal{F}}_k]_v(t  L(v))\\
& = &
\prod_{v \in \{1,\ldots,m\}^n} [\overline{\mathcal{F}}]_v(t  L(v)) \\
& \underset{m \to \infty}{\longrightarrow} &
\prod_{|v|=n} [\overline{\mathcal{F}}]_v(t  L(v)) \qquad (t \geq 0).
\end{eqnarray*}
Taking expectations on both sides, this inequality becomes an equality since the $[\overline{\mathcal{F}}]_v(t  L(v))$, $|v|=n$, are conditionally independent given $(L(v))_{|v|=n}$ and have conditional expectation $\overline{F}(t  L(v))$ by Lemma \ref{Disintegration}. This gives the asserted result.
\end{proof}

Eq.\ \eqref{DisintegratedFPE} is of essential importance for our purposes. It can be transformed
into an additive one by taking logarithms and a change of the variables $t \mapsto e^t$. To this end, fix any $\alpha>0$ and define
$$ \Psi(t)\ \gl\ e^{-\alpha t} (- \log \overline{\mathcal{F}}(e^t)) $$ 
for $t \in \mathds{R}$. Put also $S(v) \gl -\log L(v)$ for $v \in \V$ with the usual convention $S(v) \gl \infty$ on $\{L(v) = 0\}$. Then, by \eqref{DisintegratedFPE}, 
\begin{eqnarray*}
\Psi(t)
& = &
e^{-\alpha t}\left(-\log \prod_{|v|=n} [\overline{\mathcal{F}}]_v(e^t L(v)) \right) \\
& = &
\sum_{|v|=n} e^{-\alpha t} \left(-\log [\overline{\mathcal{F}}]_v(e^t L(v)) \right) \\
& = &
\sum_{|v|=n} L(v)^{\alpha} e^{-\alpha (t-S(v))} \left(-\log [\overline{\mathcal{F}}]_v(e^{t-S(v)}) \right) \\
& = &
\sum_{|v|=n} L(v)^{\alpha} [\Psi]_v(t-S(v))\quad\text{ a.s.,}
\end{eqnarray*}
that is $\Psi$ satisfies the following \emph{pathwise renewal equation}:
\begin{equation} \label{FPFRE}
\Psi(t) = \sum_{|v|=n} L(v)^{\alpha} [\Psi]_v(t-S(v))\quad\text{ a.s.}
\end{equation}
for each $t \in \mathds{R}$. To explain the notion "pathwise renewal equation", we introduce a family of measures related to \eqref{FPFRE}, namely
\begin{equation*}
\Sigma_{\alpha,n}\ \gl\ \mathds{E} \sum_{|v|=n} L(v)^{\alpha} \delta_{S(v)},\quad
n \in \mathds{N}_0.
\end{equation*}
If $m(\alpha)=1$, then $\Sigma_{\alpha,n}$ is a probability distribution on $\mathds{R}$ and
$\Sigma_{\alpha,n} = \Sigma_{\alpha,1}^{*(n)}$ (the $n$-fold convolution of $\Sigma_{\alpha,1}$) 
for each $n\ge 0$, see \textit{e.g.}\ \cite{AM2}. In the following, we denote by $(\overline{S}_{\alpha,n})_{n \geq 0}$ a random walk with increment distribution $\Sigma_{\alpha,1}$ if $m(\alpha) = 1$. The renewal measure of $(\overline{S}_{\alpha,n})_{n \geq 0}$ shall be denoted by $U$, \textit{i.e.}, $U\gl \sum_{n \geq 0} \Sigma_{\alpha,1}^{*(n)}$. Then
\begin{equation} \label{RME}
U\ =\ \sum_{n \geq 0} \Sigma_{\alpha,1}^{*(n)}
 \ =\ \sum_{n \geq 0} \Sigma_{\alpha,n}
 \ =\ \mathds{E} \sum_{v \in \V} L(v)^{\alpha} \delta_{S(v)}
 \ =\ \mathds{E} \mathcal{U},
\end{equation}
where $\mathcal{U} \gl \sum_{v \in \V} L(v)^{\alpha} \delta_{S(v)}$ denotes the \emph{random weighted renewal measure} of the branching random walk $(S(v))_{v \in \V}$, for details see again \cite{AM2}.
The connection to the associated random walk has been used by various authors in the analysis of Eq.\ \eqref{SumFP1} or the branching random walk, see \textit{e.g.}\ \cite{Bi}, \cite{DL}, \cite{Ly},  or \cite{Ik}.

Now suppose $m(\alpha) \leq 1$ and define $\psi(t) \gl \mathds{E} \Psi(t)$ ($t \in \mathds{R}$). $\psi$ is well defined due to the fact that $\Psi(t) \geq 0$ a.s. for all $t \in \mathds{R}$. By taking expectations on both sides of Eq.\ \eqref{FPFRE}, we obtain
\begin{eqnarray*}
\psi(t)
& = &
\mathds{E} \sum_{|v|=n} L(v)^{\alpha} [\Psi]_v(t-S(v)) \\
& = &
\mathds{E} \left( \mathds{E} \left[ \sum_{|v|=n} L(v)^{\alpha} [\Psi]_v(t-S(v)) \Big| \mathcal{A}_n \right] \right) \\
& = &
\mathds{E} \sum_{|v|=n} L(v)^{\alpha} \psi(t-S(v)) \\
& = &
\int \psi(t-s) \, \Sigma_{\alpha,n}(ds),
\end{eqnarray*}
having utilized that $[\Psi]_v$ is independent of $\mathcal{A}_n$ for $|v|=n$. Consequently, $\psi$ satisfies
the renewal equation
\begin{equation} \label{FPIRE}
\psi(t)\ =\ \int \psi(t-s) \, \Sigma_{\alpha,n}(ds),\quad t\in\mathds{R}
\end{equation}
of which \eqref{FPFRE} is a disintegrated version. This provides the justification for the notion \emph{"pathwise renewal equation"}. While uniqueness results for renewal equations of the form \eqref{FPIRE} are commonly known, uniqueness results for processes solving a pathwise renewal equation are systematically studied in \cite{AM2}. The following result is cited from there:

\begin{Theorem}[cf.\ \cite{AM2}] \label{PRE}
Suppose that $\mathds{E} N > 1$, $\mathds{P}(T \in \{0,1\}^{\mathds{N}}) < 1$ and $m(\alpha)=1$ for some $\alpha>0$.
Let $\Psi: \mathds{R} \times \Omega \rightarrow [0,\infty]$ denote a $\mathfrak{B} \otimes \mathcal{A}_{\infty}$-measurable stochastic process which solves Eq.\ \eqref{FPFRE} for $n=1$. Then the following assertions hold true:
\begin{itemize}
 \item[(a)]
  Suppose $\Sigma_{\alpha,1}$ is nonarithmetic. \\
  If, at each $t \in\mathds{R}$, $\Psi$ is a.s.\ left continuous with right hand limit and locally 
  uniformly integrable and if $\sup_{t \in \mathds{R}} \mathds{E} \Psi(t) < \infty$,
  then $\Psi$ is a version of $cW^{(\alpha)}$ for some $c \ge 0$.
 \item[(b)]
  Suppose $\Sigma_{\alpha,1}$ is $d$-arithmetic ($d > 0$). \\
  If $\sup_{n \in \mathds{Z}} \mathds{E} \Psi(s+nd) < \infty$ for all $s \in [0,d)$,
  then there exists a $d$-periodic function $p:\mathds{R} \rightarrow \mathds[0,\infty)$ such that $\Psi$ is a version of $pW^{(\alpha)}$.
\end{itemize}
\end{Theorem}

In order to utilize this theorem in the context of fixed-point equations, we need to check whether the additive transformation $\Psi$ of the disintegrated fixed-point $\mathcal{F}$ satisfies the assumptions of the theorem. Clearly, $\Psi$ is product measurable and standard arguments also show that it is a.s.\ left continuous with right hand limit at any $t\in\mathds{R}$. Applicability of Theorem \ref{PRE} therefore reduces
to a verification of the integrability conditions for $\Psi$. This is not always possible but works for the subclass of $\alpha$-bounded fixed points 
and forms the key ingredient to our proof of Theorem \ref{alpha-elementar-Darstellung} in Section \ref{FRE}.

\section{The characteristic exponent} \label{CharakteristischerExponent}

The purpose of this section is to provide some results related to the existence of the characteristic exponent $\alpha$ of $T$. Recall from Remark \ref{remcharexp} that a sufficient condition for this to be true is that $m(\alpha)=1$ and $\mathds{P}
(W^{(\alpha)}>0)>0$. Section \ref{cascade} is devoted to a class of examples
showing that solutions to our SFPE \eqref{MinFP1} may exist even if $T$
does not have a characteristic exponent. This will subsequently be used to point out some phenomena which may occur in the situation of Eq.\ \eqref{MinFP1} but not in the situation of its additive counterpart, \textit{i.e.}, Eq.\ \eqref{SumFP1}.

\subsection{Necessary conditions for the existence of the characteristic exponent}

\begin{Lemma} \label{Transienz}
If $m(\alpha) \leq 1$ for some $\alpha>0$, then
$$  R_{n} = \sup_{|v|=n} L(v) \longrightarrow 0\quad\text{ a.s.} \qquad (n \to \infty). $$
\end{Lemma}
\begin{proof}
Since $R_{n} \le (W_n^{(\alpha)})^{1/\alpha}$ a.s.\ for each $n\ge 0$, there is nothing
to show if $W_n^{(\alpha)} \rightarrow 0$ a.s. Hence suppose that $W^{(\alpha)}$ is nondegenerate. In this case, $m(\alpha)=1$ and the associated random walk $(\overline{S}_{\alpha,n})_{n \geq 0}$ with increment distribution $\Sigma_{\alpha,1}$ converges a.s.\ to $\infty$ by Theorem \ref{Biggins}. A well known result from renewal theory
(cf.\ \textit{e.g.}\ \cite[p.\,200ff]{Fe}) further tells us that the transience of $(\overline{S}_{\alpha,n})_{n \geq 0}$ implies $U(I) < \infty$ for all compact intervals $I \subseteq \mathds{R}$, where $U$ denotes the renewal measure of $(\overline{S}_{\alpha,n})_{n \geq 0}$ (cf.\ Section \ref{WBP}). Since $\mathds{E} \mathcal{U} = U$ by Eq.\ \eqref{RME}, we infer $\mathcal{U}(I) < \infty$ a.s.\ for all compact intervals $I$ and therefore
$$ \mathds{P}\left(\limsup_{n \to \infty} R_{n} \in \{0,\infty\}\right) = 1. $$
Finally, use $R_{n} \leq (W_n^{(\alpha)})^{1/\alpha} \rightarrow (W^{(\alpha)})^{1/\alpha}
< \infty$ a.s.\ to conclude $R_{n}\to 0$ a.s.
\end{proof}

\begin{Lemma} \label{Fquer<1}
If the characteristic exponent $\alpha$ exists, then $\overline{F}(t) < 1$ for all $t > 0$
and $F \in \mathfrak{F}_{\wedge}$.
\end{Lemma}
\begin{proof}
Suppose there exists a $t_0 > 0$ with $\overline{F}(t_0) = 1$. Pick any $t > t_0$ and let $\tau\gl\inf\{n \geq 0: R_{n} \leq t_0/t\}$. Then $\tau$ is an a.s.\ finite stopping time by Lemma \ref{Transienz}. Hence a combination of the optional sampling theorem applied to the bounded martingale $(\overline{\mathcal{F}}_{n}(t))_{n\ge 0}$ and Lemma \ref{Disintegration} yields
$$ \overline{F}(t)\ =\ \mathds{E}\overline{\mathcal{F}}_{\tau}(t)\ =
 \ \mathds{E} \prod_{|v|=\tau} \overline{F}(t L(v)) 
  \ \geq\ \mathds{E} \prod_{|v|=\tau} \overline{F}(t_0)\ =\ 1. $$
Since $t >t_0$ was chosen arbitrarily, we have $\overline{F}(t) = 1$ for all $t \geq 0$, which is clearly impossible for any proper distribution on $\mathds{R}^{+}$.
\end{proof}

\begin{Lemma} \label{WrnualphainFinf}
Suppose that $m(\alpha) = 1$ and $\mathds{E}W^{(\alpha)}=1$ for some $\alpha>0$. Then $\mathfrak{W}_{\Lambda_{\alpha}}(r,\alpha) \subseteq \mathfrak{F}_{\wedge,e}^{\alpha}$ in the $r$-geometric case ($r > 1$) and  $\mathfrak{W}_{\Lambda_{\alpha}}(1,\alpha) \subseteq \mathfrak{F}_{\wedge,e}
^{\alpha}$ in the continuous case.
\end{Lemma}
\begin{proof}
In what follows we restrict ourselves to the $r$-geometric case as the continuous case is similar but easier. So let $h \in \mathfrak{H}(r,\alpha)$ and then $F$ as in Definition \ref{Wrnualpha}(a) with $\Lambda=\Lambda_{\alpha}$. Recall that $\varphi_{\alpha}$ denotes the Laplace transform of $W^{(\alpha)}$.
Then $\overline{F}(t) = \varphi_{\alpha}(h(t)t^{\alpha})$ ($t \geq 0$). Since $h$ is multiplicatively $r$-periodic and all positive $T_i$ take values in $r^{\mathds{Z}}$ a.s., we infer
\begin{eqnarray*}
\mathds{E} \prod_{i \geq 1} \overline{F}(t T_i)
& = &
\mathds{E} \prod_{i \geq 1} \varphi_{\alpha}(h(tT_i) (tT_i)^{\alpha}) \\
& = &
\mathds{E} \prod_{i \geq 1} \varphi_{\alpha}(h(t) t^{\alpha} T_i^{\alpha})  \\
& = & 
\varphi_{\alpha}(h(t) t^{\alpha})\ =\  \overline{F}(t),
\end{eqnarray*}
having used that $\Lambda_{\alpha} \in \mathfrak{F}_{\Sigma}(T^{(\alpha)})$, which in turn follows from the equation $W^{(\alpha)} = \sum_{i \geq 1} T_i^{\alpha} [W^{(\alpha)}]_i$ a.s. Thus, $\overline{F}$ solves Eq.\ \eqref{MinFP2}, \textit{i.e.}, $F \in \mathfrak{F}_{\wedge}$. One can easily check that $F$ is always $\alpha$-elementary.
\end{proof}

\begin{Lemma} \label{OtalphaleqWalpha}
Let $F\in \mathfrak{F}_{\wedge,b}^{\alpha}$, $\mathcal{F}=(\mathcal{F}(t))_{t\ge 0}$ be its disintegration and $\alpha > 0$ be such that $m(\alpha) \leq 1$. Then there exists a positive constant $C>0$ such that 
$$ \mathds{P} \left(-\log \overline{\mathcal{F}}(t) \leq Ct^{\alpha} W^{(\alpha)}
  \text{ for all } t \geq 0 \right)\ =\ 1. $$
\end{Lemma}
\begin{proof}
By assumption, $1-\overline{F}(t) \leq Ct^{\alpha}/2$ for all sufficiently small $t \geq 0$ and some $C>0$. Using this and $-\log x \leq 2(1-x)$ for all $x$ in a suitable left neighbourhood of $1$, we infer
$$ -\log \overline{F}(t) \leq 2(1-\overline{F}(t)) \leq Ct^{\alpha} $$
for all $0 \leq t \leq \delta$ with $\delta > 0$ sufficiently small. By Lemma \ref{Transienz} in the following
section, $m(\alpha)\le 1$ ensures $\mathds{P}(B)=1$ for $B\gl\{\sup_{|v|=n} L(v) \rightarrow 0\}$. Fixing any $t > 0$, we have $t L(v) \leq \delta$ on $B$ for all $|v| \geq n$ and some sufficiently large $n$. Hence
$-\log \overline{F}(t L(v)) \leq Ct^{\alpha}L(v)^{\alpha}$  on $B$ for all $|v| \geq n$ which in turn implies
\begin{eqnarray*}
-\log \overline{\mathcal{F}}(t)
& = &
-\log \liminf_{n \to \infty} \prod_{|v|=n} \overline{F}(t L(v)) \\
& = &
\limsup_{n \to \infty} \sum_{|v|=n} -\log \overline{F}(t L(v)) \\
& \leq &
\limsup_{n \to \infty} Ct^{\alpha} \sum_{|v|=n} L(v)^{\alpha} \\
& = &
2Ct^{\alpha} W^{(\alpha)}\quad\text{a.s.}
\end{eqnarray*}
on the almost certain event $B$.
\end{proof}

\begin{Rem} \label{OtalphaleqWalphaRem} \rm
(a) Lemma \ref{OtalphaleqWalpha} allows the following obvious modification: Suppose $m(\alpha) \leq 1$ for some $\alpha>0$ and $\G(T)=r^{\mathds{Z}}$ for some $r>1$. Let $F\in \mathfrak{F}_{\wedge}$ be such that $1-\overline{F}(sr^{-n}) \leq C(sr^{-n})^{\alpha}$ for some $s,C > 0$ and all sufficiently large $n \in\mathds{N}$. Then $-\log \overline{\mathcal{F}}(sr^n) \leq 2C(sr^n)^{\alpha} W^{(\alpha)}$ a.s.\ for all $n \in \mathds{Z}$, where $\mathcal{F}$ denotes the disintegration of $F$.

(b) If $W^{(\alpha)}=0$ a.s., which is always true if $m(\alpha)<1$ and may
be true if $m(\alpha)=1$, then the assertion of Lemma
\ref{OtalphaleqWalpha} becomes $\overline{\mathcal{F}}(t)=1$ a.s.\ for all
$t\ge 0$ which implies $\overline{F}(t)=1$ for all $t\ge 0$ which is clearly impossible. Hence Lemma \ref{OtalphaleqWalpha} is really about
the case when $m(\alpha)=1$ and $\mathds{E}W^{(\alpha)}=1$.
\end{Rem}

\begin{Lemma} \label{liminflimsup}
Suppose that $m(\alpha)=1$ and $\mathds{E}W^{(\alpha)}=1$ for some $\alpha>0$.
Then the following assertions hold true for any $F\in\mathfrak{F}_{\wedge}$:
\begin{itemize}
 \item[(a)] $\limsup_{t \downarrow 0} t^{-\alpha}(1-\overline{F}(t)) > 0$.
 \item[(b)] $\liminf_{t \downarrow 0} t^{-\alpha}(1-\overline{F}(t)) < \infty$.
\end{itemize}
\end{Lemma}
\begin{proof}
(a) Suppose that $\lim_{t \downarrow 0} t^{-\alpha}(1-\overline{F}(t)) = 0$. Then
$F$ clearly satisfies the crucial assumption of Lemma \ref{OtalphaleqWalpha} 
for \emph{every} $C>0$, and we conclude for its disintegration $\mathcal{F}$ that
$\mathds{P}(-\log \overline{\mathcal{F}}(t) \leq 2 C t^{\alpha} W^{(\alpha)}
\text{ for all }t \geq 0)=1$ for every $C>0$. As a consequence, $-\log \overline{\mathcal{F}}(1) = 0$ a.s. which in turn implies $\overline{F}(1)=\mathds{E}\overline{\mathcal{F}}(1) = 1$. But this contradicts Lemma \ref{Fquer<1} and we conclude that 
$\limsup_{t \downarrow 0} t^{-\alpha}(1-\overline{F}(t)) > 0$ as claimed.

\vspace{.2cm}
(b) Suppose that $\liminf_{t \downarrow 0} t^{-\alpha}(1-\overline{F}(t)) = \infty$. 
This time we will produce a contradiction by comparison of $F(t)$ with the class
$F_{s}(t)\gl 1-\varphi_{\alpha}((t/s)^{\alpha})$, $s>0$, of solutions to \eqref{MinFP1}
(see Lemma \ref{WrnualphainFinf}). Since $\lim_{t \downarrow 0} t^{-\alpha}
(1-\overline{F}_{s}(t))=s^{-\alpha}\mathds{E}W^{(\alpha)}=s^{-\alpha}$ for any $s>0$, we infer $\overline{F}(t)\le\overline{F}_{s}(t)$ for any $s>0$ and $0<t<\varepsilon(s)$
with $\varepsilon(s)$ sufficiently small. Now fix any $t>0$ and consider the bounded
martingales $(\overline{\mathcal{F}}_{n}(t))_{n\ge 0}$ and 
$(\overline{\mathcal{F}}_{s,n}(t))_{n\ge 0}$
defined by \eqref{multimart} for $F$ and $F_{s}$, respectively. By Lemma \ref{Transienz},
the stopping time $\tau(s)\gl\inf\{n:tR_{n}<\varepsilon(s)\}$ is a.s.\ finite and
$$ \overline{\mathcal{F}}_{\tau}(t)\ =\ \prod_{|v|=\tau}\overline{F}(tL(v))\ \le\ 
\prod_{|v|=\tau}\overline{F}_{s}(tL(v))\ =\ \overline{\mathcal{F}}_{s,\tau}(t) $$
for any $s>0$. Therefore, by an appeal to the optional sampling theorem,
\begin{equation*}
\overline{F}(t)\ =\ \mathds{E}\overline{\mathcal{F}}_{\tau}(t)\ \le 
\ \mathds{E}\overline{\mathcal{F}}_{s,\tau}(t)\ =\ \overline{F}_{s}(t)
\end{equation*}
for any $s>0$. Finally, use $\overline{F}_{s}(t)=\varphi_{\alpha}((t/s)^{\alpha})\to
\varphi_{\alpha}(\infty)=\mathds{P}(W^{(\alpha)}=0)=0$ as $s\downarrow 0$ to
infer $\overline{F}(t)=0$ and thereupon the contradiction $F=\delta_{0}$ since $t>0$
was chosen arbitrarily.
\end{proof}

\subsection{A sufficient condition for the existence of the characteristic exponent}

\begin{Lemma} \label{Iksanowtrick}
Let $\alpha>0$ and $\mathfrak{F}_{\wedge,r}^{\alpha}\ne\emptyset$.
Then $m(\alpha) \leq 1$.
\end{Lemma}
\begin{proof}
Let $F\in\mathfrak{F}_{\wedge,r}^{\alpha}$ and recall that $D_{\alpha}\overline{F}(t)\gl t^{-\alpha}(1-\overline{F}(t))$ for $t > 0$, thus $c_{1}\gl\liminf_{t \downarrow 0}D_{\alpha}\overline{F}(t)>0$ and $c_{2}\gl\limsup_{t \downarrow 0}D_{\alpha}\overline{F}(t)<\infty$. It follows from Eq.\ \eqref{MinFP2} that (see equation (10) in \cite{Ik})
$$ 1\ =\ \mathds{E} \sum_{i=1}^N T_i^{\alpha} \frac{D_{\alpha}\overline{F}(t T_i)}{D_{\alpha}\overline{F}(t)} \prod_{j<i} \overline{F}(t T_j). $$
Now use Fatou's lemma to infer
\begin{eqnarray*}
1
&=&
\mathds{E} \sum_{i=1}^N T_i^{\alpha} \liminf_{t \downarrow 0}
\frac{D_{\alpha}\overline{F}(t T_i)}{D_{\alpha}\overline{F}(t)} \prod_{j<i} \overline{F}(t T_j)\\
&\ge&
\frac{c_{1}}{c_{2}}\,\mathds{E}  \sum_{i=1}^N T_i^{\alpha}
\ =\ \frac{c_{1}}{c_{2}}\,m(\alpha).
\end{eqnarray*}
But the same argument applies to Eq.\ \eqref{MinFPWBsf} (that
is \eqref{MinFP2} after $n$ iterations of the SFPE) and gives
\begin{equation*}
1\ \ge\ \frac{c_{1}}{c_{2}}\,\mathds{E}  \sum_{|v|=n}L(v)^{\alpha}
\ =\ \frac{c_{1}}{c_{2}}\,m(\alpha)^{n}
\end{equation*}
for each $n\ge 1$ and thereupon the desired conclusion.
\end{proof}

Given two sequences $(a_{n})_{n\ge 1}$, $(b_{n})_{n\ge 1}$
of real numbers, we write $a_{n}\sim b_{n}$ hereafter if $\lim_{n\to\infty}b_{n}^{-1}a_{n}=1$ and $a_{n}\asymp b_{n}$ if $0<\liminf_{n\to\infty}b_{n}^{-1}a_{n}\le
\limsup_{n\to\infty}b_{n}^{-1}a_{n}<\infty$.

\begin{Prop}\label{suffcharacexp}
Let $\alpha>0$ and $\mathfrak{F}_{\wedge,r}^{\alpha}\ne\emptyset$.
Then $\alpha$ is the characteristic exponent of $T$ and $W^{(\alpha)}$ a.s.\ positive.
\end{Prop}

\begin{proof} Note first that $m(\alpha)\le 1$ follows by Lemma \ref{Iksanowtrick} and then $R_{n}\to 0$ a.s.\ by Lemma \ref{Transienz}. Since $F\in\mathfrak{F}_{\wedge,r}^{\alpha}$ is not a Dirac measure
(cf.\ Proposition \ref{beschrSatz2}), we have
$\mathds{P}(\overline{\mathcal{F}}(t)<1)>0$ for some $t>0$. By combining these facts
with $-\log(1-x)\sim x$ as $x\downarrow 0$ and $\alpha$-regularity,
we infer
\begin{eqnarray*}
-\log\overline{\mathcal{F}}_{n}(t)
&=&
\sum_{|v|=n}-\log\overline{F}(tL(v))\\
&\sim&
\sum_{|v|=n}\left(1-\overline{F}(tL(v))\right)\\
&\asymp&
\sum_{|v|=n}L(v)^{\alpha}\ =\ W_{n}^{(\alpha)},\quad n\to\infty
\end{eqnarray*}
which in combination with $-\log\overline{\mathcal{F}}_{n}(t)\to-\log\overline{\mathcal{F}}(t)$ a.s. (Lemma \ref{Disintegration}) shows
$$ \liminf_{n\to\infty}W_{n}^{(\alpha)}\ >\ 0\quad\text{a.s.} $$
on the event $\{\overline{\mathcal{F}}(t)<1\}$. Hence $m(\alpha)=1$,
$W^{(\alpha)}=\lim_{n\to\infty}W_{n}^{(\alpha)}>0$ a.s.\ and $\mathds{E}
W^{(\alpha)}=1$ (by Theorem \ref{Biggins} and \eqref{A1}).
\end{proof}

\section{Proofs of Theorem \ref{alpha-elementar-Charakterisierung} and
Theorem \ref{alpha-elementar-Darstellung}} \label{FRE}

\subsection{Proof of Theorem \ref{alpha-elementar-Charakterisierung}}

\noindent
"(a) $\Rightarrow$ (b)" is trivial as $\mathfrak{F}_{\wedge,e}^{\alpha}
\subset\mathfrak{F}_{\wedge,r}^{\alpha}$.

\vspace{.2cm}\noindent
"(b) $\Rightarrow$ (c)" is Proposition \ref{suffcharacexp}.

\vspace{.2cm}\noindent
"(c) $\Rightarrow$ (a)". Assuming (c) we have $\mathds{E} W^{(\alpha)} = 1$ (see Theorem \ref{Biggins} in the Appendix) whence the Laplace transform $\varphi_{\alpha}$ of $W^{(\alpha)}$ satisfies $\varphi_{\alpha}'(0) = -1$. By Lemma \ref{WrnualphainFinf}, $F(t)\gl 1-\varphi_{\alpha}(t^{\alpha})$, $t \geq 0$, defines a solution to Eq.\ \eqref{MinFP1}. Furthermore,
$$ \lim_{t \downarrow 0}D_{\alpha}\overline{F}(t)
 \ =\ \lim_{t \downarrow 0}t^{-\alpha}(1-\varphi_{\alpha}(t^{\alpha}))
 \ =\ - \varphi_{\alpha}'(0) = 1, $$
and thus $F$ constitutes an $\alpha$-elementary solution to Eq.\ \eqref{MinFP1}.

\vspace{.2cm}\noindent
"(c) $\Leftrightarrow$ (d)".
As already pointed out, this is a consequence of Biggins' martingale limit theorem stated as Theorem \ref{Biggins} in the Appendix.\hfill $\square$

\subsection{Proof of Theorem  \ref{alpha-elementar-Darstellung}}

In view of Lemma \ref{WrnualphainFinf} it remains to verify that $\mathfrak{F}_{\wedge,b}^{\alpha}\subseteq\mathfrak{W}_{\Lambda_{\alpha}}(d,\alpha)$, where
$d=1$, if $\G(T)=\mathds{R}^{+}$, and $d=r$, if $\G(T)=r^{\mathds{Z}}$.
To this end, let $F$ be $\alpha$-bounded with disintegration $\overline{\mathcal{F}}$.
By Lemma \ref{OtalphaleqWalpha}, we have
\begin{equation} \label{mathcalFleqWalpha}
\mathds{P}\left(- \log \overline{\mathcal{F}}(t) \leq Ct^{\alpha} W^{(\alpha)}
 \text{ for all }t\ge 0\right)\ =\ 1
\end{equation}
for a suitable $C > 0$. Recall that $\overline{\mathcal{F}}$ satisfies the multiplicative Eq.\ \eqref{DisintegratedFPE}, which, upon logarithmic transformation and setting
$$ \Psi(t)\ \gl\ e^{-\alpha t}(-\log \overline{\mathcal{F}}(e^t)),\quad t \in \mathds{R} $$
becomes the following pathwise renewal equation (see \eqref{FPFRE}):
$$ \Psi(t)\ =\ \sum_{|v|=n} L(v)^{\alpha} [\Psi]_v(t - S(v))\quad\text{ a.s.}     $$
for all $t \in\mathds{R}$. We want to make use of Theorem \ref{PRE} and must therefore
check its conditions as for the random function $\Psi$. We already mentioned right after Theorem \ref{PRE} that $\Psi$ is product measurable and a.s.\ left-continuous with right hand limits at any $t\in\mathds{R}$. Since, by \eqref{mathcalFleqWalpha},
$$ 0\ \leq\ \Psi(t)\ \leq\ e^{-\alpha t} C e^{\alpha t} W^{(\alpha)}\ =\ C W^{(\alpha)} $$
for all $t \in \mathds{R}$ on a set of probability one, it follows that 
$\sup_{t \in \mathds{R}} \mathds{E} \Psi(t) < \infty$ and that $\Psi$ is locally uniformly integrable. Hence Theorem \ref{PRE} applies and we infer $\Psi(t) = p(t) W^{(\alpha)}$ a.s., where $p$ denotes a measurable $(\log r)$-periodic function in the $r$-geometric case and a positive constant in the continuous case, respectively. In both cases,
\begin{eqnarray*} 
\overline{\mathcal{F}}(t)
& = &
\exp(- t^{\alpha} \Psi(\log t)) \nonumber \\
& = &
\exp(- p(\log t) t^{\alpha} W^{(\alpha)}) \nonumber \\
& = &
\exp(- h(t) t^{\alpha} W^{(\alpha)})
\text{ a.s.}
\end{eqnarray*}
for all $t > 0$, where $h(t) \gl p(\log t)$. Taking expectations on both sides of this equation provides us with
$$ \overline{F}(t) = \varphi_{\alpha}(h(t) t^{\alpha}) \qquad (t>0). $$
Therefore, the proof is complete in the continuous case where $h$ is necessarily constant. For the rest of the proof, suppose we are in the $r$-geometric case. Then $p$ is $(\log r)$-periodic as mentioned above and thus $h$ multiplicatively $r$-periodic. Furthermore, the left continuity of $\overline{F}$ implies the left continuity of the function $t \mapsto h(t) t^{\alpha}$ and, therefore, the left continuity of $h$. Similarly, we conclude that $t \mapsto h(t) t^{\alpha}$ is nondecreasing. These facts together give $h \in \mathfrak{H}(r,\alpha)$, which finally shows  $F \in \mathfrak{W}_{\Lambda_{\alpha}}(1,\alpha)$.\hfill $\square$

\vspace{.5cm}
A combination of Theorem \ref{alpha-elementar-Darstellung} and Lemma
\ref{liminflimsup} provides us with a very short proof of the following result
about the distribution $\Lambda_{\alpha}$ of $W^{(\alpha)}$ as a solution
to \eqref{SumFP1} with $T^{(\alpha)}$ instead of $T$:

\begin{Cor} \label{nualphaunique}
Suppose \eqref{A1} and that $m(\alpha) = 1$ and $\mathds{E}W^{(\alpha)} = 1$ for some $\alpha > 0$. Then $\mathfrak{F}_{\Sigma}(T^{(\alpha)})=
\{\Lambda_{\alpha}(c\,\cdot):c>0\}$, \textit{i.e.}, $\Lambda_{\alpha}$ is the unique
solution up to scaling to \eqref{SumFP1} for $T^{(\alpha)}$.
\end{Cor}

This result (with $\mathds{E}N>1$ instead of condition \eqref{A1})
has been obtained under slightly stronger conditions on $m(\cdot)$
by Biggins and Kyprianou as Theorem 1.5 in \cite{BK1} and as Theorem 3 in \cite{BK3}, where the latter result also covers the boundary case briefly discussed in Remark \ref{commentsthm}(c).

\begin{proof}
Suppose $\Lambda \in \mathfrak{F}_{\Sigma}(T^{(\alpha)})$ is another
solution to \eqref{SumFP1} for $T^{(\alpha)}$ with Laplace transform $\psi$. Regard $\psi(t^{\alpha})$ as the survival function of a probability measure $G$. Then $G \in \mathfrak{F}_{\wedge}$ and $\lim_{t \downarrow 0}D_{\alpha}
\overline{G}(t)=\lim_{t \downarrow 0} t^{-\alpha}(1-\psi(t^{\alpha})) = |\psi'(0)|\in\mathds{R}^{+}$, the finiteness following from Lemma \ref{liminflimsup}. Consequently,
$G \in \mathfrak{F}_{\wedge,b}^{\alpha}$ and thus $\psi(t) = \varphi_{\alpha}(h(t) t)$ for some $h \in \mathfrak{H}(r,1)$ in the $r$ geometric case or a constant $h$ in the continuous case. In both cases, $|\psi'(0)| = \lim_{t \downarrow 0} t^{-1}(1-\psi(t)) = \lim_{t \downarrow 0} h(t)$ implies $h= |\psi'(0)|$, \textit{i.e.}, $\psi (t)= \varphi_{\alpha}(t/c)$ with $c\gl 1/|\psi'(0)|>0$. This proves $\Lambda = \Lambda_{\alpha}(c\,\cdot)$.
\end{proof}

\section{Beyond $\alpha$-boundedness: The generalized water cascades example} \label{cascade}

In the following we shall discuss a class of examples which demonstrates that
nontrivial solutions to the SFPE \eqref{MinFP1} may exist even if $T$
does not have a characteristic exponent. This is in contrast to the additive
case, \textit{i.e.}, Eq.\ \eqref{SumFP1}, for which the existence of the characteristic exponent and the existence of nontrivial solutions are
equivalent, at least under appropriate conditions on $T$ and $N$, see Durrett and Liggett \cite{DL} and Liu \cite{Liu}.

We fix $N \in \mathds{N}$, $N \geq 2$, and denote by $B_1, \ldots, B_N$ independent Bernoulli variables with parameter $\vartheta \in (0,1)$, that is $\mathds{P}(B_i = 1) = \vartheta = 1 - \mathds{P}(B_i = 0)$, $i=1,\ldots,N$. Put $T_i \gl e^{-B_i}$ for $i=1, \ldots, N$ and $T_i = 0$ for $i > N$. Notice that
$\G(T)=e^{\mathds{Z}}$ ($e$-geometric case).
Then Eq.\ \eqref{MinFP1} takes the form
\begin{equation} \label{JagersRoeslerBeispielFPGleichung}
X \stackrel{d}{=} \min_{1\le i\le N} \frac{X_i}{T_i}.
\end{equation}
For $N=2$ and $\vartheta>1/2$, this example was studied by
Jagers and R\"osler \cite{JR}.

\begin{Lemma} \label{kritischLemma}
In the situation of Eq.\ \eqref{JagersRoeslerBeispielFPGleichung} the following assertions are equivalent:
\begin{itemize}
 \item[(a)]
  The characteristic exponent $\alpha > 0$ exists.
 \item[(b)]
  $\vartheta > 1 - 1/N$.
 \item[(c)]
  If $\tilde{N} \gl \sum_{i=1}^N \mathbbm{1}_{\{B_i = 0\}}=
  \sum_{i=1}^N \mathbbm{1}_{\{T_i = 1\}}$, then
  the Galton-Watson process with offspring distribution
  $\mathds{P}(\tilde{N} \in \, \cdot)$ is subcritical.
\end{itemize}
\end{Lemma}
\begin{proof}

Under stated assumptions, we have
$$ m(\alpha)\ =\ \mathds{E} \sum_{i=1}^N T_i^{\alpha}\ =\ N (\vartheta e^{-\alpha} + (1-\vartheta)). $$
Therefore $N(1-\vartheta) < 1$ is necessary and sufficient for the existence of a positive $\alpha$ such that $m(\alpha) = 1$. This proves that (a) and (b) are equivalent. The equivalence of (b) and (c) is obvious as
$\mathds{E} \tilde{N} = N (1-\vartheta)$.
\end{proof}

According to Lemma \ref{kritischLemma}, we distinguish three cases:
\begin{itemize} 
\item[{(1)}] \emph{Subcritical case}: $\vartheta > 1 - 1/N$.
\vspace{-.2cm}
\item[{(2)}] \emph{Critical case}: $\vartheta = 1 - 1/N$.
\vspace{-.2cm}
\item[{(3)}] \emph{Supercritical case}: $\vartheta < 1 - 1/N$.
\end{itemize}
Consider the associated WBP as introduced in Section \ref{WBP}. Now define
$$ \tilde{L}(v)\ \gl\ \mathbbm{1}_{\{L(v)=1\}} $$
for $v \in \V_N\gl \bigcup_{n \geq 0} \{1,\ldots,N\}^n$.
By our model assumptions, $\tilde{G}_n \gl \sum_{|v|=n} \tilde{L}(v)$ ($n \geq 0$) defines a Galton-Watson process with offspring distribution $\mathds{P}(\tilde{N} \in \, \cdot)$. In the supercritical case, $(\tilde{G}_n)_{n \geq 0}$ survives with positive probability, thus
\begin{equation} \label{exceptional}
\mathds{P}(\limsup_{n \to \infty} R_{n} = 1)\ =\ \mathds{P}(\tilde{G}_n \text{ survives}) > 0. 
\end{equation}
The main outcome of the subsequent discussion will be that Eq.\ \eqref{JagersRoeslerBeispielFPGleichung} has nontrivial solutions in all three possible cases.  In view of \eqref{exceptional} in the supercritical case, this shows that the existence of a nontrivial solution to \eqref{MinFP1} does not necessarily entail $\limsup_{n \to \infty} R_{n} = 0$ a.s.\ as intuition might suggest.

We proceed with the construction of nontrivial solutions to equation 
\eqref{JagersRoeslerBeispielFPGleichung} and start by taking a look at the associated functional equation. Given a solution $F$, the latter takes the form
\begin{equation*}
\overline{F}(t)\ =\ \mathds{E} \prod_{i=1}^N \overline{F}(t T_i)
\ =\ \left(\vartheta \overline{F}(te^{-1}) + (1-\vartheta) \overline{F}(t)\right)^N,
\quad t\ge 0,
\end{equation*}
which, upon solving for $\overline{F}(te^{-1})$, leads to
\begin{equation} \label{Fte^-1}
  \overline{F}(te^{-1}) 
 \ =\ \vartheta^{-1}\left(\overline{F}(t)^{1/N} - (1-\vartheta) \overline{F}(t)\right)
 \ =\ g_{N,\vartheta}(\overline{F}(t)),
\end{equation}
where $g _{N,\vartheta}(u) \gl \vartheta^{-1}\left(u^{1/N} - (1-\vartheta)u \right)$ for $0 \leq u \leq 1$. The following lemma collects some properties of $g_{N,\vartheta}$.

\begin{Lemma} \label{gNthetaDiskussion}
The following assertions hold true for $g _{N,\vartheta}$:
\begin{itemize}
 \item[(a)]
  $\{u \in [0,1]:g _{N,\vartheta}(u) = u\}=\{0,1\}$.
 \item[(b)]
  $g _{N,\vartheta}(u) > u$ for all $u \in (0,1)$.
 \item[(c)]
  If $\vartheta \geq 1-1/N$ (subcritical or critical case), then 
  $g _{N,\vartheta}$ is strictly increasing on $(0,1)$. 
  In particular, $g_{N,\vartheta}(u) \in (0,1)$ for all $u \in (0,1)$.
 \item[(d)]
  If $\vartheta < 1-1/N$ (supercritical case), there exists a unique 
  $a_0 \in (0,1)$ satisfying $g _{N,\vartheta}(a_0) = 1$.
  $g_{N,\vartheta}$ is strictly increasing on $[0,a_0]$ and $>1$
  on $(a_{0},1)$.
\end{itemize}
\end{Lemma}
\begin{proof}
Obviously, $g_{N,\vartheta}(u) = u$ holds iff $u^{1/N} = u$ and thus iff
$u \in \{0,1\}$, for $N\ge 2$. This shows (a). Next, $g_{N,\vartheta}(u) > u$ for all sufficiently small $u>0$ because 
$g_{N,\vartheta}$ is continuously differentiable on $(0,1]$ with $\lim_{u \downarrow 0} g_{N,\vartheta}'(u) = \infty$. But this gives (b), by the continuity of $g_{N,\vartheta}$ and (a), and we also infer that $g_{N,\vartheta}'$ is positive in a right neighborhood of $0$. Furthermore, $g_{N,\vartheta}'(u) = \vartheta^{-1}\left(N^{-1} u^{1/N - 1} - (1-\vartheta) \right)$, $u\in (0,1]$, implies that $g_{N,\vartheta}'(u) = 0$ for $u > 0$ iff $u^{1/N - 1} = N(1-\vartheta)$. Hence $g_{N,\vartheta}'(u) \not = 0$ on $(0,1)$ in the subcritical and critical case ($N(1-\vartheta) \leq 1$), and (c) is true. In the supercritical case ($N(1-\vartheta) > 1$), we have $g_{N,\vartheta}'(a) = 0$ for a unique $a\in (0,1)$. Consequently, $g_{N,\vartheta}$ is strictly increasing on $(0,a)$ and strictly decreasing on $(a,1)$. Since $g_{N,\vartheta}(1)=1$, (d) must be true (\textit{cf.}\ also
Figure 1).
\end{proof}

\begin{figure}[htbp]
 \begin{center}
  \includegraphics[width=8.1cm]{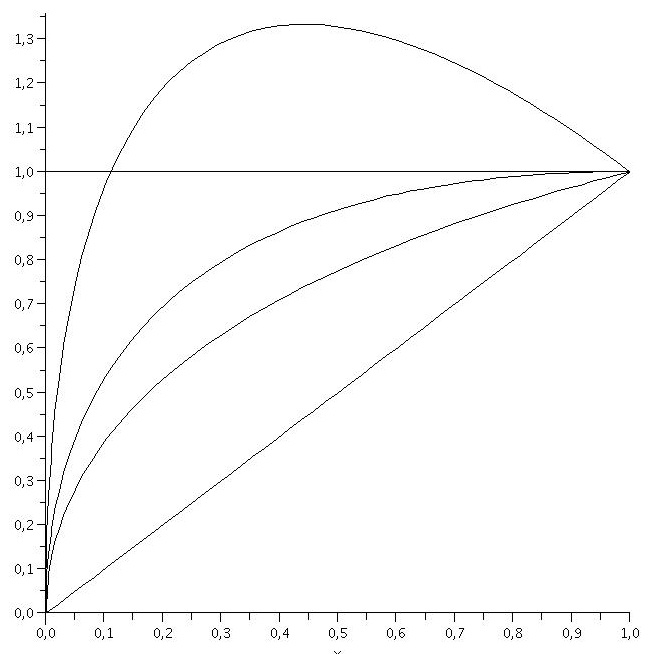}
  \caption{$g_{N,\vartheta}$ for $N=2$ and $\vartheta = 1/4, 1/2, 3/4$
       and the identity function (from top to bottom)}
 \end{center}
 \label{fig:gGraphik}
\end{figure}

\noindent
{\bf A. Critical and subcritical case}. Assuming $\vartheta \geq 1-1/N$,
we have, by Lemma \ref{gNthetaDiskussion}, that $g_{N,\vartheta}$ is strictly increasing with unique fixed points $0$ and $1$ in the unit interval. Therefore, its inverse function, denoted by $g_{N,\vartheta}^{\circ(-1)}$, exists on $[0,1]$. We can rewrite Eq.\ \eqref{Fte^-1} in terms of $g_{N,\vartheta}^{\circ(-1)}$ as
$$ \overline{F}(t) = g_{N,\vartheta}^{\circ(-1)}(t/e),\quad t \geq 0. $$
Equations of this type have been completely solved in Theorem 2.1 of \cite{AM1}, and its application allows us here to provide a full description of $\mathfrak{F}_{\wedge}$. To this end, let $g_{N,\vartheta}^{\circ(n)}$ denote the $|n|$-fold composition of $g_{N,\vartheta}$ ($n \geq 1$) or its inverse $g_{N,\vartheta}^{\circ(-1)}$ ($n \leq -1$), and let $g_{N,\vartheta}^{\circ(0)}$ be the identity function.
Although the situation in \cite{AM1} differs slightly from ours, we adopt the notation from there and write $\mathcal{F}_+$ for the set of nonincreasing, left continuous functions $f:(1,e] \rightarrow (0,1)$, satisfying $f(t) \leq g_{N,\vartheta}(f(e))$ for all $t \in (1,e]$. Then the following theorem is proved
along the same lines as Theorem 2.1 of \cite{AM1} with only minor changes
in obvious places.

\begin{Theorem}[see \cite{AM1}, Theorem 2.1] \label{Maximin}
If $\vartheta \geq 1-1/N$ (subcritical or critical case), there is a one-to-one
correspondence $F\leftrightarrow f$ between $\mathfrak{F}_{\wedge}$ and $\mathcal{F}_+$, established by
\begin{equation} \label{(sub-)critical}
\overline{F}(t) = g_{N,\vartheta}^{\circ(-n)} f(t/e^n),
\end{equation}
where $n \in \mathds{Z}$ denotes the unique integer satisfying $1 < t/e^n \leq e$.
\end{Theorem}

In the subcritical, where the characteristic exponent exists,
we can state the following interesting corollary.

\begin{Cor} \label{Finf=Wralpha}
Suppose $\vartheta > 1-1/N$ (subcritical case) and let $\alpha$ be the characteristic exponent, \textit{i.e.}, $m(\alpha)=1$. Then $\mathfrak{F}_{\wedge} = \mathfrak{W}_{\Lambda_{\alpha}}(e,\alpha)$. In particular, any $F\in \mathfrak{F}_{\wedge}$ can be written as $F(t) = 1-\varphi_{\alpha}(h(t) t^{\alpha})$, $t > 0$ for a unique $h \in \mathfrak{H}(e,\alpha)$.
\end{Cor}
\begin{proof}
By Proposition \ref{WrnualphainFinf}, $\mathfrak{W}_{\Lambda_{\alpha}}(e,\alpha) \subseteq \mathfrak{F}_{\wedge}$. Conversely, fix any $F\in \mathfrak{F}_{\wedge}$ and put $h(t) \gl \varphi_{\alpha}^{\circ(-1)}(\overline{F}(t)) t^{-\alpha}$ ($1< t \leq e$), where $\varphi_{\alpha}^{\circ(-1)}:(0,\infty) \rightarrow (0,1)$ denotes the inverse function of $\varphi_{\alpha}$, the Laplace transform of $W^{(\alpha)}$ (note that $0<\overline{F}(t)<1$ for all $t>0$ by Lemma \ref{Fquer<1} and Remark \ref{CompactSupport}). Extend $h$ to a multiplicatively $e$-periodic function on $(0,\infty)$. Then it can easily be seen that $h \in \mathfrak{H}(e,\alpha)$. Thus, $G(t)\gl 1-\varphi_{\alpha}(h(t) t^{\alpha})$ ($t > 0$) defines an element of $\mathfrak{W}_{\Lambda_{\alpha}}(e,\alpha)$ and thus an element of $\mathfrak{F}_{\wedge}$ as well. Moreover, $F$ and $G$ coincide on $(1,e]$ and, as a consequence of Theorem \ref{Maximin}, $F$ and $G$ are uniquely determined by their values on $(1,e]$. Hence $F = G \in \mathfrak{W}_{\Lambda_{\alpha}}(e,\alpha)$.
\end{proof}

\noindent
{\bf B. Supercritical case}. Assuming now $\vartheta<1-1/N$, we follow an idea of 
R\"osler and Jagers (\cite{JR}, 2004) and construct a particular nontrivial (and discrete) solution $F$ to Eq.\ \eqref{JagersRoeslerBeispielFPGleichung} which is then shown to be unique up to scaling (Theorem \ref{superkritischSatz}). The first step is to define a sequence $(a_n)_{n \geq 0}$ such that $\{0,1-a_{0},1-a_{1},\ldots\}$ forms the range of $\overline{F}$. Let $a_{0}$ be the unique (by Lemma \ref{gNthetaDiskussion}(d)) value in $(0,1)$ such that $g_{N,\vartheta}(a_0) = 1$. Lemma \ref{gNthetaDiskussion} also ensures that $g_{N,\vartheta}$ is strictly increasing on $[0,a_0]$ and $g_{N,\vartheta}(u) > u$ for $u \in (0,a_0]$. Hence $g_{N,\vartheta}^{\circ(-1)}:[0,1]\to [0,a_0]\subset [0,1]$ as well as its iterations are well defined, and we
we can choose $a_n \gl g_{N,\vartheta}^{\circ(-n)}(a_0)$ for $n \geq 1$. Evidently, $(a_n)_{n \geq 0}$ constitutes a decreasing sequence of positive numbers and thus $a_{\infty} \gl \lim_{n \to \infty} a_n$ exists. Since $g_{N,\vartheta}$ is continuous, we have $g_{N,\vartheta}(a_{\infty}) = a_{\infty}$ so that $a_{\infty} = 0$ by Lemma \ref{gNthetaDiskussion}(a). We now define our candidate $F$ as
\begin{equation} \label{JRLoesung}
F(t)\ \gl\
  \begin{cases}
  0 & \text{ if } 0 \leq t \leq 1, \\
  1-a_n & \text{ if } e^n < t \leq e^{n+1}\text { for } n \in \mathds{N}_0
  \end{cases}
\end{equation}
and note that $F$ is clearly left continuous and increasing with $\lim_{t \to \infty} \overline{F}(t) = 1$ and therefore a distribution function. Moreover, for $n \in \mathds{N}_0$ and  $e^n < t \leq e^{n+1}$,
$$ \overline{F}(te^{-1})\ =\ a_{n-1} = g_{N,\vartheta} (a_n) = g_{N,\vartheta}(\overline{F}(t)), $$
where $a_{-1}\gl1$. As $\overline{F}(t) = 1$ for $t \leq 1$, this shows that $\overline{F}$ solves \eqref{Fte^-1} and so $F\in\mathfrak{F}_{\wedge}$.

\begin{Theorem} \label{superkritischSatz}
If $\vartheta < 1-1/N$ (supercritical case), then 
\begin{equation*}
\mathfrak{F}_{\wedge}\ =\ \{G\in\mathcal{P}:G(t) = F(ct)\text{ for all }t\ge 0\text{ and
some }c>0\},
\end{equation*}
where $F$ is given by \eqref{JRLoesung}.
\end{Theorem}
\begin{proof}
We have already proved that $F\in \mathfrak{F}_{\wedge}$. Also, $\mathfrak{F}_{\wedge}$ is obviously closed under scaling, \textit{i.e.}, $\mathfrak{F}_{\wedge} \supseteq \{\overline{F}(c \, \cdot): \, c>0\}$. Conversely, let $G\in \mathfrak{F}_{\wedge}$ and notice that $G$ cannot be concentrated in a single point by Proposition \ref{beschrSatz2}. Hence we can choose $t > 0$ such that $\overline{G}(t) \in (0,1)$. Now suppose $\overline{G}(t) \notin \{a_n: n \geq 0\}$. Then there exists a unique $n \ge 0$ satisfying $a_n < \overline{G}(t) < a_{n-1}$ (where $a_{-1} := 1$). Then from the definition of the $a_{n}$ we  obtain
\begin{equation*}
\overline{G}(te^{-(n+1)}) = g_{N,\vartheta}^{\circ(n+1)}(\overline{G}(t)) 
\in g_{N,\vartheta}((a_0,a_{-1})) \subseteq (1,\infty),
\end{equation*}
which is obviously impossible.
Thus, $\overline{G}(t) \in \{1\} \cup \{a_n: \, n \geq 0\}$ for all $t \geq 0$. Finally,
put $c\gl\sup\{t\ge 0:\overline{G}(t)=1\}$, thus $\overline{G}(c)=1$, and use once again Eq.\ \eqref{Fte^-1} and the recursive definition of the $a_{n}$ to obtain $G=F(c\,\cdot)$. Further details are omitted.
\end{proof}

\section{Related results for the smoothing transformation} \label{ConRem}

We have already pointed out in the Introduction that equations \eqref{MinFP1} and \eqref{SumFP1}, and thus also the associated maps $M$ in \eqref{MapFP} and $M_{\Sigma}$ in \eqref{MapST},  are naturally connected via the functional equation (see \eqref{MinFP2})
$$ g(t)\ =\ \mathds{E}\prod_{i\ge 1}g(tT_{i}),\quad t\ge 0, $$
valid for all survival functions $g=\overline{F}$ of solutions $F$ to
\eqref{MinFP1} and for Laplace transforms $g=\varphi$ of solutions to
\eqref{SumFP1}. The connection is even closer owing to the fact that any
Laplace transform vanishing at $\infty$ is also the survival function of a distribution on $[0,\infty)$. It is therefore not surprising that our results stated in Section 4 have corresponding versions for the additive case dealing with the smoothing transform $M_{\Sigma}$ and its fixed points. The latter has been studied in a large number of articles, see \textit{e.g.}\ \cite{DL}, \cite{Liu}, \cite{Cal2}, \cite{Ik}, \cite{BK3} and \cite{AR2}.

In order to formulate the counterparts of Theorem \ref{alpha-elementar-Charakterisierung} and Theorem \ref{alpha-elementar-Darstellung} for
$M_{\Sigma}$, we first recall from \cite{DL} the definition of $\alpha$-stable laws and their $r$-periodic extensions which take here the role of the Weibull distributions in the min-case.

\begin{Def}[\textit{cf.}\ \cite{DL}, p.\,280] \label{periodicstable} \rm
For $\alpha > 0$ and $r>1$, let $\mathfrak{P}(r,\alpha)$ be the set of multiplicativley $r$-periodic functions $p:\mathds{R}^{+} \rightarrow \mathds{R}^{+}$ such that $p(t)t^{\alpha}$ has a completely monotone derivative.
Then, for $p \in \mathfrak{P}(r,\alpha)$, the \emph{$r$-periodic $\alpha$-stable law} $r$-$\mathcal{S}(p,\alpha)$ is defined as the distribution on $[0,\infty)$ with Laplace transform $\varphi(t) = e^{-p(t) t^{\alpha}}$ ($t > 0$).
\end{Def}

Note that $\varphi$ really defines a Laplace transform by Criterion 2 on page 441 of \cite{Fe}, for $t \mapsto p(t)t^{\alpha}$ is positive with completely monotone derivative. Furthermore, if $p(t)\equiv c$, then $\varphi$ is the Laplace transform of a positive stable law with scale parameter $(c/\cos(\pi \alpha/2))^{1/\alpha}$, shift $0$ and index of stability $\alpha$ (\textit{cf.}\ \cite{ST}, Definition 1.1.6 and Proposition 1.2.12). This distribution, denoted as $\mathcal{S}(c,\alpha)$, does not depend on $r$ (which is thus dropped in
the notation). For our convenience, we also define $r$-$\mathcal{S}(0,\alpha) =\mathcal{S}(0,\alpha) \gl \delta_0$.

We continue with a short proof of the known fact that periodic stable laws as defined above only exist for $\alpha\le 1$.

\begin{Lemma} \label{stable}
In the situation of Definition \ref{periodicstable}, any element of $\mathfrak{P}(r,1)$ is constant, while $\mathfrak{P}(r,\alpha)=\emptyset$ for any $\alpha
>1$.
\end{Lemma}

\begin{proof}
If $\alpha \geq 1$, we have for each $s \in (1,r]$
$$ p(s)
  = \lim_{n \to \infty} \frac{1-e^{-p(sr^{-n}) (sr^{-n})^{\alpha}}}{(sr^{-n})^{\alpha}}
  = \lim_{n \to \infty} \frac{1-e^{-p(sr^{-n}) (sr^{-n})^{\alpha}}}{(sr^{-n})} \cdot (sr^{-n})^{1-\alpha}. $$
Since $e^{-p(t)t^{\alpha}}$ is convex and $t^{1-\alpha}$ nonincreasing if
$\alpha\ge 1$, we see that $p$ must be nonincreasing in $s$. In fact, $p$ is even strictly decreasing if $\alpha > 1$. But in view of the periodicity of $p$, 
the latter is impossible, thus $\mathfrak{P}(r,\alpha)=\emptyset$ for
$\alpha>1$, while $p$ must be constant in the case $\alpha = 1$.
\end{proof}

Our next definition is just the canonical modification of Definition \ref{Wrnualpha} for the previously defined generalized stable laws.

\begin{Def} \label{Srnualpha} \rm
Let $\alpha\in (0,1]$ and $\Lambda$ be a probability measure on $[0,\infty)$.
Then
\begin{itemize}
 \item[(a)]
 $\mathcal{S}_{\Lambda}(1,\alpha)$ denotes the class of 
 $\Lambda$-mixtures of positive $\alpha$-stable laws $F$ of the form
 $ F(\cdot) = \int \mathcal{S}(yc,\alpha)(\cdot) \, \Lambda(dy),  $
  where $c > 0$.
 \item[(b)]
 $\mathcal{S}_{\Lambda}(r,\alpha)$ for $r>1$ denotes the class of   $\Lambda$-mixtures of $r$-$\mathcal{S}(p,\alpha)$ distributions   $F$ of the form
 $ F(\cdot) = \int r\text{-}\mathcal{S}(yp,\alpha)(\cdot)\,\Lambda(dy), $
 where $p \in \mathfrak{P}(r,\alpha)$.
\end{itemize}
\end{Def}

Finally, the notions "$\alpha$-bounded", "$\alpha$-regular" and "$\alpha$-elementary" for fixed points of $M_{\Sigma}$ are defined exactly as in Definition \ref{fixptype} for fixed points of $M$, when substituting $D_{\alpha}\overline{F}$ with $t^{-\alpha}(1-\varphi(t))$, $\varphi$ the Laplace transform
of $F$. The respective classes are denoted as $\mathfrak{F}_{\Sigma,b}^{\alpha},\mathfrak{F}_{\Sigma,r}^{\alpha}$ and $\mathfrak{F}_{\Sigma,e}^{\alpha}$, respectively.

We are now ready to formulate the results corresponding to our theorems
in Section \ref{MainResults} for the smoothing transform. The standing assumptions \eqref{A1}
and \eqref{A2} for the min-case can be replaced here with the weaker ones 
\begin{equation}  \tag{A3}\label{A3}
 \text{Supercriticality:}\hspace{3.8cm}\mathds{E}N>1;
 \hspace{1cm}
\end{equation}
\begin{equation}  \tag{A4}\label{A4}
 \text{Nondegeneracy:}\hspace{1.8cm}\mathds{P}(T\in\{0,1\}  ^{\infty}\}<1.\hspace{1cm}
\end{equation}

\begin{Theorem} \label{alpha-elementar-Charakterisierung2}
Suppose \eqref{A3} and \eqref{A4}. Then the following assertions are equivalent for any $\alpha>0$:
\begin{itemize}
 \item[(a)]  Equation \eqref{SumFP1} has an $\alpha$-elementary      solution
    $(\mathfrak{F}_{\Sigma,e}^{\alpha}\ne\emptyset)$.
 \item[(b)]  Equation \eqref{SumFP1} has an $\alpha$-regular solution
    $(\mathfrak{F}_{\Sigma,r}^{\alpha}\ne\emptyset)$.

 \item[(c)]  $\alpha\in (0,1]$, $m(\alpha)=1$ and $\mathds{P}(W^{(\alpha)}>0)>0$.
 \item[(d)]  $\alpha\in (0,1]$, $m(\alpha)=1$, the random walk $     (\overline{S}_{\alpha,n})_{n \geq 0}$  converges to      $\infty$ a.s.\ and
    $$ \int_{(1,\infty)} \left[\frac{u \log u}
    {\mathds{E}(\overline{S}_{\alpha,1}^+ \wedge \log u)} \right] 
    \mathds{P} (W_1^{(\alpha)} \in \, du) < \infty. $$
\end{itemize}
\end{Theorem}

\begin{Theorem} \label{alpha-elementar-Darstellung2}
Suppose \eqref{A3}, \eqref{A4} and that $m(\alpha)=1$ and $\mathds{E}W^{(\alpha)}=1$ for some $\alpha\in (0,1]$. Then
\begin{equation*}
 \mathfrak{F}_{\Sigma,b}^{\alpha} = \mathfrak{F}_{\Sigma,r}^{\alpha}  = \mathfrak{F}_{\Sigma,e}^{\alpha} = \mathcal{S}_{\Lambda_{\alpha}} (d,\alpha),
 \end{equation*}
where $d=r>1$ in the $r$-geome\-tric case $(\G(T)=r^{\mathds{Z}})$
and $d=1$ in the continuous case $(\G(T)=\mathds{R}^{+})$.
\end{Theorem}

The proofs of the previous two theorems are essentially the same as for Theorems \ref{alpha-elementar-Charakterisierung} and \ref{alpha-elementar-Darstellung} except for the additional assertion $\alpha \leq 1$. But the arguments in the proof of Theorem \ref{alpha-elementar-Charakterisierung} show that the Laplace transform $\varphi$ of any element of $\mathfrak{F}_{\Sigma,b}^{\alpha}$ can be written as $\varphi(t) = \varphi_{\alpha}(p(t) t^{\alpha})$ ($t > 0$), where $p$ is a multiplicatively $r$-periodic function in the $r$-geometric case and a constant in the continuous case. Owing to Lemma \ref{stable} we get $\alpha \leq 1$ if we can prove that $p \in \mathfrak{P}(r,\alpha)$ in the $r$-geometric case. By writing $p(t) = \varphi^{\circ(-1)} (\varphi_{\alpha}(t)) \cdot t^{-\alpha}$ ($t > 0$) (where $\varphi^{\circ(-1)}$ denotes the inverse function of $\varphi$), we see that $p$ is infinitely often differentiable. It remains to verify that $t \mapsto p(t) t^{\alpha}$ has a completely monotone derivative. To this end, we observe that
\begin{eqnarray*}
r^{\alpha n} (1-\varphi(t r^{-n}))
& = &
\frac{1-\varphi_{\alpha}(p(t r^{-n}) (t r^{-n})^{\alpha})}{p(t r^{-n}) \cdot (t r^{-n})^{\alpha}} \cdot p(t) t^{\alpha} \\
& \to &
p(t) t^{\alpha} \qquad (n \to \infty),
\end{eqnarray*}
for $n \in \mathds{N}$
having utilized that $\varphi_{\alpha}'(0) = - \mathds{E} W^{(\alpha)} = -1$. Since $t \mapsto r^{\alpha n} (1-\varphi(t r^{-n}))$ has a completely monotone derivative for each $n \in \mathds{N}$ and since the convergence is uniform on compact sets, $t \mapsto p(t)t^{\alpha}$ has a completely monotone derivative as well.

\begin{appendix}
\section{Appendix: Biggins' Theorem} 

The assumptions \eqref{A3} and \eqref{A4} are in force throughout. Let $q\in [0,1)$
denote the extinction probability of $W_{n}^{(0)}=\sum_{|v|=n}\mathds{1}_{\{L(v)
>0\}},\,n\ge 0$ (an ordinary supercritical Galton-Watson process if $N=W_{1}
^{0)}<\infty$ a.s.).

The subsequent characterization theorem for martingale limits in branching random walks (which are nothing but limits of WBP's having the martingale property) is a crucial ingredient to our analysis of $\alpha$-elementary fixed points. In the stated most general form, which imposes no conditions on $T$ beyond $m(\alpha)=1$, it was recently obtained by Alsmeyer and Iksanov \cite{AI}, but the first version of the result under additional assumptions on $T$ was obtained more than three decades ago by Biggins \cite{Bi} and later reproved (under slightly relaxed conditions) by Lyons \cite{Ly} via a measure-change argument (size-biasing). Alsmeyer and Iksanov combined Lyons' argument with Goldie and Maller's results on perpetuities \cite{GM}. 

\begin{Theorem}[\textit{cf.}\ \cite{AI}, Theorem 1.4] \label{Biggins}
Suppose \eqref{A3}, \eqref{A4}, and $m(\alpha)=1$. Then the following four assertions are equivalent:
\begin{itemize}
 \item[(a)]  $\mathds{P}(W^{(\alpha)} > 0) > 0$.
 \item[(b)]  $\mathds{P}(W^{(\alpha)} > 0) = q$.
 \item[(c)]  $\mathds{E} W^{(\alpha)} = 1$.
 \item[(d)]  $\lim_{n \to \infty} \overline{S}_{\alpha,n} = \infty$ a.s. and
    $$ \int_{(1,\infty)} \left[\frac{u \log u}
    {\mathds{E}(\overline{S}_{\alpha,1}^+ \wedge \log u)} \right] 
    \mathds{P} (W_1^{(\alpha)} \in \, du) < \infty. $$
\end{itemize}
\end{Theorem}

We further state the following corollary which says that $m(\alpha)=1$ and $\mathds{P}(W^{(\alpha)}>0)>0$ always imply that $\alpha$ equals the characteristic exponent of $T$, that is, the \emph{minimal} value at which $m$ equals 1. This is relevant to be pointed out because
$m$, as a strictly convex function on $\{\beta\ge 0:m(\beta)<\infty\}$,
may equal 1 for two values $\alpha_{1},\alpha_{2}$.

\begin{Cor} \label{charcharexp}
Suppose $m(\alpha)=1$ and $\mathds{P}(W^{(\alpha)}>0)>0$.
Then $m(\beta)>1$ for all $\beta\in [0,\alpha)$ so that $\alpha$ is the characteristic exponent of $T$.
\end{Cor}

\begin{proof} First note that $R_{n}=\max_{|v|=n}L(v)\to 0$ a.s.\ follows from Lemma \ref{Transienz}, which is easily seen to remain valid under \eqref{A3} and \eqref{A4}.
Now, if $m(\alpha)=1$ and $\mathds{P}(E)>0$ for $E\gl\{W^{(\alpha)}>0\}$, we see
that
$$ W_{n}^{(\alpha-\eta)}\ \ge\ R_{n}^{-\eta}W_{n}^{(\alpha)}\ \to\ \infty\quad
\text{a.s.\ on }E $$
for any $\eta\in (0,\alpha]$ which is clearly impossible if $m(\alpha-\eta)\le 1$,
where $(W_{n}^{(\alpha-\eta)})_{n\ge 0}$ would be a nonnegative supermartingale.
\end{proof}

\end{appendix}


\begin{thebibliography}{999}
 \bibitem{AldBan} {\sc Aldous, D., Bandyopadhyay, A.} (2005). 
 A survey of max-type recursive distributional 
 equations. \emph{Ann. Appl. Probab.} \textbf{15}, 1047-1110.
 \bibitem{AldSt}  {\sc Aldous, D., Steele, J.M.} (2003). 
 The objective method: Probabilistic combinatorial optimization 
 and local weak convergence. 
 In \emph{Probability on Discrete Structures} (Encyclopaedia    of Mathematical Sciences \textbf{110}), ed. H. Kesten, 1-72.
 Springer, New York.
 \bibitem{AI}  {\sc Alsmeyer, G., Iksanov, A.} (2009).
 A log-type moment result for perpetuities and its
 application to martingales in the supercritical branching  
 random walk. {\em Electronic J.\ Probab.} {\bf 14}, 289-313.
 \bibitem{AM1}  {\sc Alsmeyer, G., Meiners, M.} (2007). 
 A stochastic fixed-point equation related to game-tree     evaluation. \emph{J. Appl. Probab.} \textbf{44}, 586-606.
  \bibitem{AR1}  {\sc Alsmeyer, G., R\"osler, U.} (2006). 
 A stochastic fixed-point equation related to weighted branching   with determinstic weights.  \emph{Electronic J. Probab.}      \textbf{11}, 27-56.
 \bibitem{AR2}  {\sc Alsmeyer, G., R\"osler, U.} (2008). 
 A stochastic fixed-point equation related        to weighted minima and maxima. 
 \emph{Ann. Inst. H. Poincar\'{e}} \textbf{44}, 89-103.
 \bibitem{Bi}  {\sc Biggins, J.D.} (1977). 
 Martingale convergence in the branching random walk.
 \emph{J.  Appl. Probab.} \textbf{14}, 25-37.
 \bibitem{BK1}  {\sc Biggins, J.D., Kyprianou, A.E.} (1997). 
 Seneta-Heyde norming in the branching random walk.
 \emph{Ann. Probab.} \textbf{25}, 337-360
 \bibitem{BK3}  {\sc Biggins, J.D., Kyprianou, A.E.} (2005).
  Fixed points of the smoothing transform: the boundary case. 
  \emph{Electronic J. Probab.} \textbf{10}, 609-631.
 \bibitem{Cal2} {\sc Caliebe, A.} (2004). Representation of fixed points of a  smoothing transformation. \emph{Mathematics and Computer Science III}, eds. M. Drmota et al., Trends in         Math., Birkh\"auser, Basel, 311-324.
  \bibitem{Dev}  {\sc Devroye} (2001). 
 On the probabilistic worst-case time of ``Find''.
 \emph{Algorithmica} \textbf{31}, 291-303.   
 \bibitem{DL}  {\sc Durrett, R., Liggett, T.M.} (1983). 
 Fixed points of the smoothing transformation.      
  \emph{Z. Wahrsch. verw. Gebiete} \textbf{64}, 275-301.
 \bibitem{Fe}  {\sc Feller, W.} (1971). 
 \emph{An Introduction to Probability Theory and its Applications},
 Vol. II, 2nd Edition, Wiley, New York.
 \bibitem{GM}  {\sc Goldie, C., Maller, R.} (2000).
 Stability of perpetuities. \emph{Ann. Probab.}
 \textbf{28} (2000), 1195-1218.
 \bibitem{Gui}  {\sc Guivarc'h, Y.} (1990).
 Sur une extension de la notion semi-stable.
 \emph{Ann. Inst. H. Poincar\'{e}} \textbf{26}, 261-285. 
 \bibitem{Ik}  {\sc Iksanov, A.} (2004). 
 Elementary fixed points of the BRW smoothing transforms with    infinite number of summands. 
 \emph{Stoch. Proc. Appl.} \textbf{114}, 27-50.
 \bibitem{JR}  {\sc Jagers, P., R\"osler, U.} (2004)
 Stochastic fixed points for the maximum.        \emph{Mathematics and Computer Science III}, 325-338. Trends
 Math. Birkh\"auser, Basel.
 \bibitem{Liu}  {\sc Liu, Q.} (1998). 
 Fixed points of a generalized smoothing transformation and    applications to  the branching random walk.
 \emph{Adv. Appl. Probab.} \textbf{30}, 85-112.
 \bibitem{Ly}  {\sc Lyons, R.} (1997). 
 A simple path to Biggins' martingale convergence for branching random  walk. In \emph{Classical and Modern Branching Processes}  (IMA Vol.   Math. Appl. {\bf 84}), eds. K.B. Athreya and P. Jagers, 217-221.
 Springer, New York.
 \bibitem{AM2}  {\sc Meiners, M.} (2009). 
 Weighted branching and a pathwise renewal equation. To appear in \emph{Stoch.\ Proc.\ Appl.}
 \bibitem{NR1}  {\sc Neininger, R., R\"uschendorf, L.} (2005).
 Analysis of algorithms by the contraction method: additive and max-recursive sequences.
 \bibitem{R1}  {\sc R\"uschendorf, L.} (2006).
 On stochastic recursive equations of sum and max type.
 \emph{J. Appl. Probab.} \textbf{43}, 687-703.
 \bibitem{ST}  {\sc Samorodnitsky, G., Taqqu, M.S.} (1994).
 \emph{Stable Non-Gaussian Random Processes: Stochastic Models with Infinite Variance}. Chapman \& Hall, New York.
\end{thebibliography}
\end{document}